\documentclass[10pt,reqno]{amsart}
\usepackage[margin=0.8in]{geometry}
\usepackage{amsthm, amsmath,amsfonts,amssymb,euscript,hyperref,graphics,color,slashed}
\usepackage{graphicx}
\usepackage{mathrsfs}
\usepackage{comment}
\usepackage{latexsym}
\usepackage[makeroom]{cancel}

\def\inte#1{
\displaystyle\mathop{#1\kern0pt}^\circ }

\let\d=\delta
\let\e=\varepsilon
\let\z=\zeta
\let\lam=\lambda

\let\f=\frac

\let\p=\psi

\let\D=\Delta
\let\Lam=\Lambda

\let\wt=\widetilde

\def\virgp{\raise 2pt\hbox{,}}
\def\cdotpv{\raise 2pt\hbox{;}}

\def\eqdefa{\buildrel\hbox{\footnotesize def}\over =}

\def\C{\mathop{\mathbb C\kern 0pt}\nolimits}
\def\DD{\mathop{\mathbb D\kern 0pt}\nolimits}
\def\EE{\mathop{{\mathbb E \kern 0pt}}\nolimits}
\def\K{\mathop{\mathbb K\kern 0pt}\nolimits}
\def\N{\mathop{\mathbb N\kern 0pt}\nolimits}
\def\Q{\mathop{\mathbb Q\kern 0pt}\nolimits}
\def\R{\mathop{\mathbb R\kern 0pt}\nolimits}
\def\SS{\mathop{\mathbb S\kern 0pt}\nolimits}
\def\ZZ{\mathop{\mathbb Z\kern 0pt}\nolimits}
\def\TT{\mathop{\mathbb T\kern 0pt}\nolimits}
\def\P{\mathop{\mathbb P\kern 0pt}\nolimits}

\newcommand{\Z}{{\ZZ}}

\def\curl{\mathop{\rm curl}\nolimits}

\def\na{\nabla}
\def\p{\partial}

\newcommand{\beq}{\begin{equation}}
\newcommand{\eeq}{\end{equation}}
\newcommand{\ben}{\begin{eqnarray}}
\newcommand{\een}{\end{eqnarray}}
\newcommand{\beno}{\begin{eqnarray*}}
\newcommand{\eeno}{\end{eqnarray*}}

\newcommand{\vv}[1]{\boldsymbol{#1}}

\def\curl{\text{curl}\,}

\newtheorem*{Main Theorem}{Main Theorem}
\newtheorem{theorem}{Theorem}[section]
\newtheorem{lemma}{Lemma}[section]
\newtheorem{proposition}{Proposition}[section]
\newtheorem{corollary}{Corollary}[section]
\newtheorem{definition}{Definition}[section]
\newtheorem{remark}{Remark}[section]

\numberwithin{equation}{section}

\begin{document}
\title[Well-posedness]{ Long time existence for the Boussinesq-Full dispersion systems}

\author{Jean-Claude Saut}
\address{Laboratoire de Math\' ematiques, UMR 8628\\
Universit\' e Paris-Saclay, Paris-Sud et CNRS\\ 91405 Orsay, France}
\email{jean-claude.saut@u-psud.fr}

\author[Li XU]{Li Xu}
\address{School of Mathematics and Systems Science, Beihang University\\  100191 Beijing, China}
\email{xuliice@buaa.edu.cn}

\date{September16, 2019}
\maketitle

\vspace{1cm}
 \textit{Abstract}.

 We establish the long time existence of solutions  for the "Boussinesq-Full dispersion" systems modeling the propagation of
internal waves in a two-layer system. For the two-dimensional Hamiltonian case $b=d>0,\,a\leq0,\,c<0$, we study the global existence of small solutions of the corresponding system. \\


\tableofcontents

\setcounter{equation}{0}
\section{Introduction}
This paper is concerned with a class of asymptotic models of internal waves arising in the so-called two-layer system.
This idealized system,
when it is at rest, consists of a homogeneous fluid of depth $d_1$
and density $\rho_1$ lying over another homogeneous fluid of depth
$d_2$ and density $\rho_2 > \rho_1$.  The bottom on which both
fluids rest is presumed to be horizontal and featureless while the
top of fluid 1 is restricted by the rigid lid assumption, which is
to say, the top is viewed as an impenetrable, bounding surface. Both of these
require that the deviation of the interface be a graph over the flat
bottom, actually parametrized by a scalar function $\zeta$, see Figure below.

    \includegraphics[width=12cm]{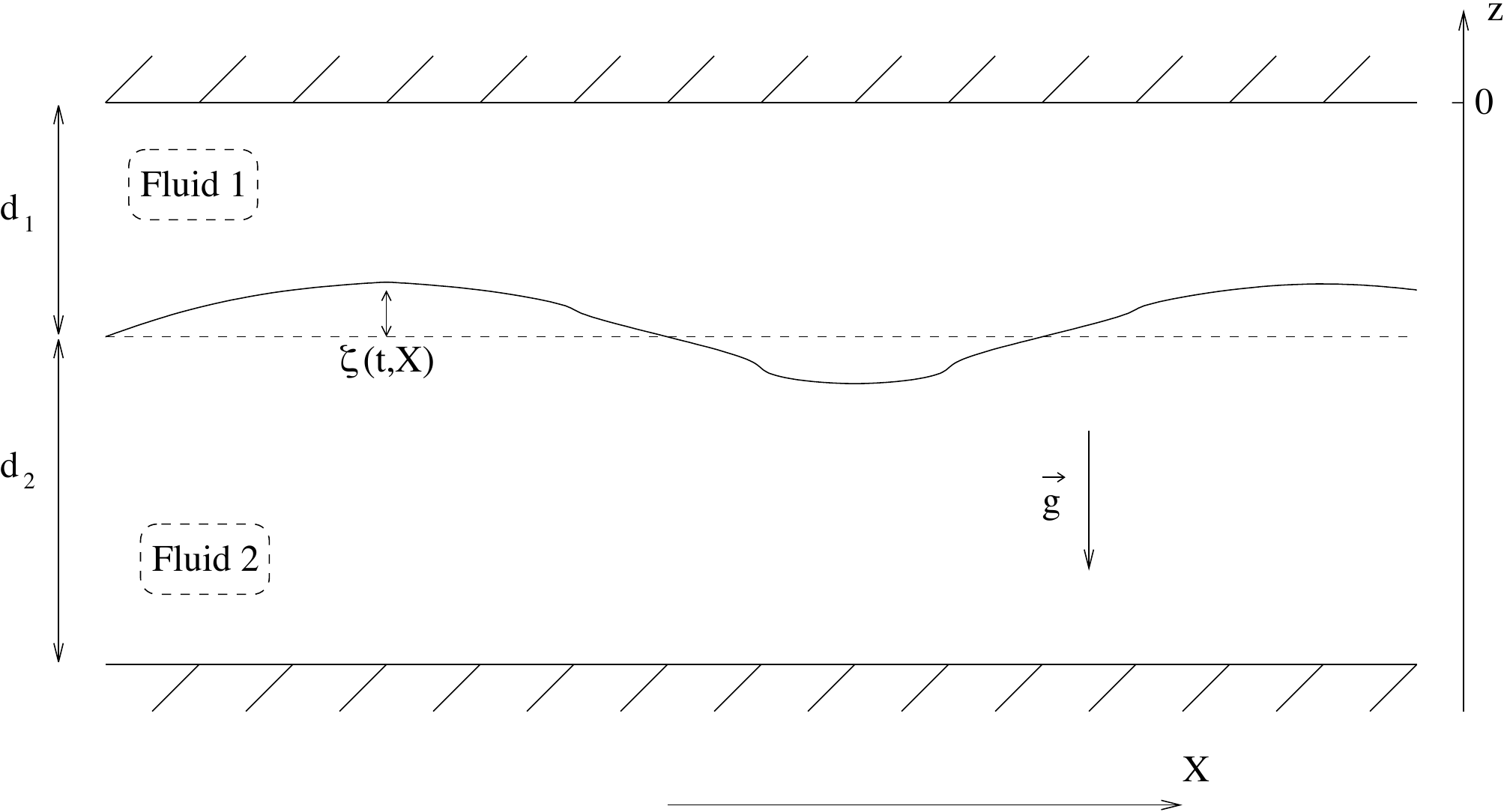}

The mathematical analysis of the full two-layer system displays tremendous difficulties due in particular to the possible appearance of Kelvin-Helmholtz instabilities.  We refer to \cite{Lannes2} for a deep analysis and far reaching results.

By expanding two non-local operators and for various ranges of parameters related to wave lengths, amplitudes, depths, densities,.., asymptotic models were rigorously (in the sense of consistency)  derived in \cite{CGK} and  \cite{BLS}. We will follow  the framework of \cite{BLS}.

More precisely, denoting $\rho_1,\rho_2$ the densities, $a$ a typical amplitude of the wave and $\lambda$ a typical wave length in the horizontal direction,  we define the dimensionless parameters

$$
    \gamma:=\frac{\rho_1}{\rho_2},\quad
    \delta:=\frac{d_1}{d_2},\quad
    \epsilon:=\frac{a}{d_1},\quad
    \mu:=\frac{d_1^2}{\lambda^2}.
$$
Though they are redundant, it is also notationally convenient to
introduce two other parameter`s $\epsilon_2$ and $\mu_2$ defined as
$$
    \epsilon_2=\frac{a}{d_2}=\epsilon\delta,\qquad
    \mu_2=\frac{d_2^2}{\lambda^2}=\frac{\mu}{\delta^2}.
$$
\begin{remark}
    The parameters $\epsilon_2$ and $\mu_2$ correspond to $\epsilon$ and
    $\mu$ with $d_2$ rather than $d_1$
    taken as the unit of length in the vertical direction.
\end{remark}

The Boussinesq-Full dispersion regime corresponds to $\mu\sim \epsilon\ll 1, \mu_2\sim 1$ so that the wave amplitude is small only with respect to the upper layer.

It is shown in \cite{BLS}    that in this Boussinesq-Full dispersion regime and in the absence of surface tension, the
two-layers system  is consistent with the \emph{three-parameter family} of Boussinesq/FD systems
\begin{equation}\label{eqB-FD}
\left\lbrace
\begin{array}{l}
(1-b\mu\Delta)\partial_t \zeta+
                  \frac{1}{\gamma}\nabla\cdot\big((1-\epsilon\zeta){\vv v}_\beta\big)\vspace{1mm}\\
     \indent-\frac{\sqrt{\mu}}{\gamma^2}|D|\coth(\sqrt{\mu_2}|D|)\nabla\cdot{\vv v}_\beta
+\frac{\mu}{\gamma}\Big(a-\frac{1}{\gamma^2}\coth^2(\sqrt{\mu_2}|D|)\Big)\Delta\nabla\cdot{\vv
v}_\beta = 0\vspace{1mm}\\
                (1-d\mu\Delta)\partial_t{\vv v}_\beta +(1-\gamma)\nabla\zeta-\frac{\epsilon}{2\gamma}\nabla(\vert{\vv v}_\beta\vert^2)
    +c\mu(1-\gamma)\Delta\nabla\zeta=0,
            \end{array}\right.
\end{equation}
where $\zeta$ is the elevation of the wave, $\gamma\in(0,1)$, ${\vv v}_\beta=(1-\beta\mu\Delta)^{-1}{\vv v}$ ($\vv v$ being the horizontal velocity) and the
constants $a$, $b$, $c$ and $d$ are defined as
    $$
    a=\frac{1}{3}(1-\alpha_1-3\beta),\qquad
    b=\frac{1}{3}\alpha_1,\qquad
    c=\beta\alpha_2,\qquad
    d=\beta(1-\alpha_2),
    $$
    with $\alpha_1\geq 0$, $\beta\geq 0$ and $\alpha_2\leq 1$.

    Note that the parameters $a,b,c,d$ are constrained by the relation $a+b+c+d=\frac{1}{3}.$


The initial condition for \eqref{eqB-FD} is imposed as follows
\beq\label{initial data}
\z|_{t=0}=\z_0,\quad{\vv v}_\beta|_{t=0}=\vv v_0,
\eeq

It is easily checked that \eqref{eqB-FD} is linearly well posed when
\beq\label{sorts of parameters}
a\leq 0, c\leq 0, b\geq 0, d\geq 0.
\eeq

The local well-posedness of the Cauchy problem for  \eqref{eqB-FD} was established in \cite{CTA2} in the following cases

\vspace{0.3cm}
 \begin{itemize}
\item[(1)]  $b>0,d>0, a\leq 0, c<0$;
\item[(2)]   $b>0 ,d>0, a\leq 0, c=0$;
\item[(3)]  $b=0, d>0, a\leq 0, c=0$;
\item[(4)] $b=0, d>0, a\leq 0, c<0$;
\item[(5)]  $b>0, d=0, a\leq 0, c=0$,
\end{itemize}

It turns out that \eqref{eqB-FD} is hamiltonian when $b=d.$ This fact has been used in \cite{AS} in the one dimensional Hamiltonian case to establish the global existence of small solutions,  by an easy extension of a similar result for the Boussinesq systems in \cite{BCS2}. We will go back to this issue for the two-dimensional Hamiltonian systems at the end of the paper.

The existence of one dimensional solitary waves for the Boussinesq -Full Dispersion systems in the Hamiltonian case was  proven in \cite{AS}. No such result  seems to be known in the non-hamiltonian case.

In the present paper we will prove the {\it long time existence} for \eqref{eqB-FD}-\eqref{initial data} that is existence on time scales of order $1/\epsilon$ for all cases stated in \eqref{sorts of parameters}. This time scale is the one on which the Boussinesq-Full Dispersion systems are "good" approximations of the two-layer system in the relevant regime.

Similar results for the "abcd" class of Boussinesq systems were established in \cite{MSZ,SX,SWX,Bu,Bu2}. As in \cite{SX, SWX} the proof of our main result is based on the derivation of a suitable symmetrizer.

 In the two-dimensional Hamiltonian case
$b=d>0,\,a\leq0,\,c<0$, we shall moreover establish the global existence of small solutions  of  \eqref{eqB-FD}-\eqref{initial data} when $\epsilon =1.$ This is as far as we know the first global existence result for this type of systems in the two-dimensional case. Similar results in the non-Hamiltonian case are not known, even in the one-dimensional case.

Before presenting the main results of this paper, we give the following definition of the functional spaces that will be used :
\begin{definition}\label{def1}
For any $s\in\R$, $k\in\N$, $\mu\in(0,1)$, the Banach space $X_{\mu^k}^{s}(\R^n)$ is defined as $H^{s+k}(\R^n)$ equipped with the norm:
\beno
\|u\|_{X_{\mu^k}^{s}}^2=\|u\|_{H^s}^2+\mu^k\|\na^ku\|_{H^{s}}^2.
\eeno
\end{definition}

The solutions to the Cauchy problem of \eqref{eqB-FD} will belong to some space $X_{\mu^k}^{s}(\R^n)\times X_{\mu^{k'}}^{s}(\R^n)$ with $k$ and $k'$ determined by $a,b,c,d$ as follows :

\begin{definition}\label{def2}
For any $a,b,c,d$ satisfying \eqref{sorts of parameters},
 we define a pair of numbers $(k,k')\eqdefa(k(a,b,c,d),k'(a,b,c,d) )$  according to the admissible sets of $(a,b,c,d)'s$ as follows:

 \vspace{0.3cm}
 \begin{itemize}
\item[(1)] $(k,k')=(3,3)$ for $b\neq d,\, b,d>0,\, a\leq 0,\,c<0$;
\item[(2)] $(k,k')=(2,2)$ for $b=d>0,\, a\leq0,\,c<0$ or $b>0,\, d=0,\, a\leq0, c=0$;
\item[(3)] $(k,k')=(4,3)$ for $b>0,\,d=0,\, a\leq 0,\,c<0$;
\item[(4)] $(k,k')=(1,2)$ for $b>0,\,d>0,\, a\leq0,\, c=0$;
\item[(5)] $(k,k')=(3,4)$ for $b=0,\, d>0,\, a\leq0,\,c<0$;
\item[(6)] $(k,k')=(1,3)$ for $b=0,\,d>0, a\leq0,\, c=0$;
\item[(7)] $(k,k')=(1,1)$ for $b=d=0,\,a\leq0,\,c<0$;
\item[(8)] $(k,k')=(0,1)$ for $b=d=0, a\leq0,c=0$.
\end{itemize}
\end{definition}

\begin{remark}
The cases (7) and (8) cannot occur for internal waves without surface tension but may occur for internal waves with a sufficiently large surface tension parameter.
\end{remark}

We now state the main results of this paper. The first theorem  concerns the long time existence for \eqref{eqB-FD}-\eqref{initial data}.
\begin{theorem}\label{Long time existence for B-FD}
Let $t_0>\f{n}{2}$, $n=1,2$, $s\geq t_0+2$  and  $a,b,c,d$ satisfy the condition \eqref{sorts of parameters}.
Assume that $\z_0\in X^s_{\mu^k}(\R^n),\vv v_0\in X^s_{\mu^{k'}}(\R^n)$ satisfy the (non-cavitation) condition
\beq\label{non-cavi condition}
1-\e\z_0\geq H>0,\quad H\in(0,1),
\eeq
where $(k,k')$ is defined in Definition \ref{def2}.
Then there exist positive constants $\wt{\epsilon}$ and $\wt{\mu}$ (maybe depending on $\|\z_0\|_{X^s_{\mu^{k}}}+\|\vv v_0\|_{X^s_{\mu^{k'}}}$ and $H$),  such that for any $\epsilon\leq\wt{\epsilon}$ and $\mu\leq\wt{\mu}$, there
exists $T>0$ independent of $\epsilon$ and $\mu$, such that \eqref{eqB-FD}-\eqref{initial data} has a unique solution $(\z,{\vv v}_\beta)$ with
$(\z,{\vv v}_\beta)\in C([0,T/\epsilon];X^s_{\mu^k}(\R^n)\times X^s_{\mu^{k'}}(\R^n))$.
Moreover,
\beq\label{long time estimate}
\max_{t\in[0,T/\epsilon]} (\|\z\|_{X^s_{\mu^k}}+\|{\vv v}_\beta\|_{X^s_{\mu^{k'}}})\leq \tilde c (\|\z_0\|_{X^s_{\mu^k}}+\|\vv v_0\|_{X^s_{\mu^{k'}}}).
\eeq
Here  $\tilde c=C(H^{-1})$ are nondecreasing functions of their argument.
\end{theorem}

\begin{remark}
In Theorem \ref{Long time existence for B-FD}, we only use the condition $\mu\ll1,\,\epsilon\ll1,\,\mu_2\sim1$. We do not need the restriction $\mu\sim\epsilon$.
\end{remark}

The second theorem is about the global existence for \eqref{eqB-FD}-\eqref{initial data}  in the Hamiltonian case $b=d>0$, $a\leq0,\, c<0$.

\begin{theorem}\label{Global existence for B-FD in case b=d}
Let $b=d>0$, $a\leq0,\, c<0$. Assume that $\z_0\in X^0_{\mu}(\R^2),\vv v_0\in X^0_{\mu}(\R^2)$. Then there exist a positive constant $\epsilon_0$(maybe depending on $\|\z_0\|_{X^s_{\mu^{k}}}+\|\vv v_0\|_{X^s_{\mu^{k'}}}$),  such that for any $\epsilon\leq\epsilon_0$ and $\mu\sim\epsilon$, \eqref{eqB-FD}-\eqref{initial data} has a unique solution $(\z,{\vv v}_\beta)$ with
$(\z,{\vv v}_\beta)\in C([0,\infty);X^0_{\mu}(\R^2)\times X^0_{\mu}(\R^2))$.
Moreover,
\beq\label{global estimate}
\max_{t\in[0,\infty)} (\|\z\|_{X^0_{\mu}}+\|{\vv v}_\beta\|_{X^0_{\mu}})\leq C (\|\z_0\|_{X^0_{\mu}}+\|\vv v_0\|_{X^0_{\mu}}).
\eeq
Here  $C$ is a universal constant which may change from line to line.
\end{theorem}

\begin{corollary}
Theorem \ref{Global existence for B-FD in case b=d} is in fact a global existence result for small solutions of  \eqref{eqB-FD} with $\epsilon\sim \mu\sim 1$ when $b=d>0$, $a\leq0,\, c<0$.
Actually one reduces to this modified system with $\epsilon=\mu =1$ by the change of variables

$$\zeta (t,X)=\epsilon^{-1}\tilde\zeta({\sqrt\mu}^{-1}t,\sqrt{\mu}^{-1}X),\quad {\vv v}_\beta(t,X)
=\epsilon^{-1}\tilde{{\vv v}}_\beta({\sqrt\mu}^{-1}t,\sqrt{\mu}^{-1}X),\; X=(x,y).$$

\end{corollary}

\setcounter{equation}{0}
\section{Preliminary}
\subsection{Notations}
Vectors will be denoted in bold letters, {\it e.g.}  $\vv v$ . When $\mathcal{B}$ is a Banach space, $\vv v\in\mathcal{B}$ means that each component of $\vv v$ belongs to $\mathcal{B}$. For any $s\in\R$, we denote by $H^s(\R^n)$ the classical  $L^2$ based Sobolev spaces with the norm $\|\cdot\|_{H^s}$.
The notation $\|\cdot\|_{L^p}$ stands for the $L^p(\R^n)$ norm, $1\leq p \leq \infty$. The $L^2(\R^n)$ inner product is denoted by
$(\vv u\,|\,\vv v)_2\eqdefa\int_{\R^n}\vv u\cdot\vv vdx$.
For any $k\in\N$,  we denote
\beno
\|f\|_{W^{k,\infty}}=\sum_{j=0}^k\|\na^jf\|_{L^\infty}.
\eeno
The notation $f\sim g$ means that there exists a constant $C$ such that $\f{1}{C}f\leq g\leq Cf$.  $f\lesssim g$ and $g\gtrsim f$ means that there exists a constant $C$ such that $f\leq Cg$. The condensed notation
$
A_s=B_s+\langle C_s\rangle_{s>\underline{s}},
$
is to say that $A_s=B_s$ if $s\leq\underline{s}$ and $A_s=B_s+C_s$ if $s>\underline{s}$.

The Fourier transform of a tempered distribution $u\in\mathcal{S}'$ is denoted by $\widehat{u}$, which is defined as follows
\beno
\widehat{u}(\xi)\eqdefa\mathcal{F}(u)(\xi)=\int_{\R^n}e^{ix\cdot\xi}u(x)dx.
\eeno
We use $\mathcal{F}^{-1}(f)$ to denote the inverse Fourier  transform of $f(\xi)$.

If $f$ and $u$ are two functions defined on $\R^n$, the  Fourier multiplier  $f(D)u$  is defined in term of Fourier transforms, i.e.,
\beno
\widehat{f(D)u}(\xi)=f(\xi)\hat{u}(\xi).
\eeno

We shall use notations
\beno
\langle\xi\rangle\eqdefa(1+|\xi|^2)^{\f12},\quad\Lambda\eqdefa\bigl(1-\D\bigr)^{\f12}.
\eeno

If  $A, B$ are two operators, $[A,B]=AB-BA$ denotes their commutator.

\vspace{0.3cm}

Throughout the paper, $a,b,c,d\in\R$, $\gamma\in(0,1)$, $\mu_2\sim1$ are given parameters. We shall not show the dependence on such given parameters. $C$ will always denote a universal constant which may be different from line to line but is independent of the parameters involved(say, $\mu$ and $\epsilon$). Otherwise, one uses the notation $C(\lam_1,\lam_2,\cdots)$ to denote a nondecreasing function of the parameters $\lam_1,\lam_2,\cdots$.

\vspace{0.3cm}

The paper is organized as follows. In the remaining part of this Section we prove technical lemmas that will be used in the proofs of the main theorems. Section 3 is devoted to the proof of Theorem \ref{Long time existence for B-FD} which involves the  symmetrization techniques used in our previous work \cite {SX} (see also \cite{SWX} on the Boussinesq $(abcd)$ systems). In  Section 4, we prove Theorem \ref{Global existence for B-FD in case b=d} by adapting the proof of a similar result for the Hamiltonian Boussinesq systems (see \cite{BCS2, Hu}). Finally an Appendix is devoted to the proof of the equivalence of norms \eqref{equivalent 1}, \eqref{equivalent 2} and \eqref{equivalent 3}.

\subsection{Symmetrizer of \eqref{eqB-FD}}
Here and in  the following sections, we shall only deal with the two-dimensional case, since the one-dimensional case is very similar  and actually  much simpler. For simplicity, we shall use $\vv v$ instead of $\vv v_\beta$ and use the following notation
\beq\label{notation}
\sigma(D)\eqdefa\sqrt{\mu_2}|D|\coth(\sqrt{\mu_2}|D|),\quad A(D)\eqdefa1+a\mu\D+\f{1}{\gamma}\sqrt{\f{\mu}{\mu_2}}\sigma(D)
+\f{1}{\gamma^2}\f{\mu}{\mu_2}\sigma(D)^2.
\eeq
With such notations, we rewrite \eqref{eqB-FD} as
\beq\label{B-FD}\left\{\begin{aligned}
&(1-b\mu\D)\p_t\z+\f{1}{\gamma}\na\cdot\bigl((A(D)-\epsilon\z)\vv v\bigr)=0,\\
&(1-d\mu\Delta)\p_t\vv v+(1-\gamma)(1+c\mu\D)\na\z-\f{\epsilon}{2\gamma}\na(|\vv v|^2)=\vv 0.
\end{aligned}\right.\eeq

If $b>0, d\geq0$ or $b=d=0$, letting $g(D)=(1-b\mu\Delta)(1-d\mu\Delta)^{-1}$, setting $\vv V=(\z,\vv v)^T=(\z,v^1,v^2)^T$, then \eqref{B-FD} is equivalent after applying $g(D)$ to the second equation to the condensed system
\beq\label{B-FD 1}
(1-b\mu\Delta)\p_t\vv V+M(\vv V,D)\vv V=\vv 0,
\eeq
where
\beq\label{M 1}\begin{aligned}
M(\vv V,D)=\left(
\begin{array}{ccc}
-\f{\epsilon}{\gamma}\vv v\cdot\na &\f{1}{\gamma}(A(D)-\epsilon\z)\p_1&\f{1}{\gamma}(A(D)-\epsilon\z)\p_2\\
(1-\gamma)g(D)(1+c\mu\Delta)\p_1&-\f{\epsilon}{\gamma}g(D)(v^1\p_1)&-\f{\epsilon}{\gamma}g(D)(v^2\p_1)\\
(1-\gamma)g(D)(1+c\mu\Delta)\p_2&-\f{\epsilon}{\gamma}g(D)(v^1\p_2)&-\f{\epsilon}{\gamma}g(D)(v^2\p_2)
\end{array}\right).\end{aligned}\eeq

When $a,b,c,d$ satisfies the condition \eqref{sorts of parameters}, the system \eqref{B-FD 1} could be treated similarly to  a symmetrizable hyperbolic system under some smallness assumption on $\epsilon$ and $\mu$. The key point to solve \eqref{B-FD 1} is to search a symmetrizer $S_{\vv V}(D)$ of $M(\vv V,D)$ such that the principal part of $iS_{\vv V}(\xi)M(\vv V,\xi)$ is self-adjoint, and that of $S_{\vv V}(\xi)$ is positive and self-adjoint under a smallness assumption on $\epsilon$ and $\mu$. It is not difficult to find that:
\begin{itemize}
\item [(i)] if $b=d$, $g(D)=1$, $S_{\vv V}(D)$ is defined by
\beq\label{S for b=d>0}\begin{aligned}
\left(
\begin{array}{ccc}
\gamma(1-\gamma)(1+c\mu\Delta)&-\epsilon v^1&-\epsilon v^2\\
-\epsilon v^1&A(D)-\epsilon\z&0\\
-\epsilon v^2&0&A(D)-\epsilon\z
\end{array}\right);\end{aligned}\eeq

\item [(ii)] if $b\neq d$, $S_{\vv V}(D)$ is defined by
\beq\label{S for b neq d}\begin{aligned}
\gamma(1-\gamma)\left(
\begin{array}{ccc}
\gamma(1-\gamma)(1+c\mu\Delta)^2g(D)&-\epsilon g(D)\bigl(v^1(1+c\mu\Delta)\bigr)&-\epsilon g(D)\bigl(v^2(1+c\mu\Delta)\bigr)\\
-\epsilon g(D)\bigl(v^1(1+c\mu\Delta)\bigr)&(A(D)-\epsilon\z)(1+c\mu\Delta)&0\\
-\epsilon g(D)\bigl(v^2(1+c\mu\Delta)\bigr)&0&(A(D)-\epsilon\z)(1+c\mu\Delta)
\end{array}\right)\\
+\epsilon^2
\left(
\begin{array}{ccc}
0&0&0\\
0&v^1v^1&v^1v^2\\
0&v^1v^2&v^2v^2
\end{array}\right)(g(D)-1).
\end{aligned}\eeq
\end{itemize}

Note that $S_{\vv V}(D)$ is not self-adjoint since at least its diagonal part is not.

\vspace{0.3cm}
Next we define the energy functional associated to \eqref{B-FD 1} as
\beq\label{E 1}
E_s(\vv V)=\bigl((1-b\mu\Delta)\Lam^s\vv V\,|\,S_{\vv V}(D)\Lam^s\vv V\bigr)_2.
\eeq
We shall show (see Appendix) that $E_s(\vv V)$ defined in \eqref{E 1} is actually a true energy functional equivalent to some $X_{\mu^k}^s(\R^2)$ norm.

\begin{remark}
When $b=0, d>0$, \eqref{B-FD} is equivalent after applying $(1-d\mu\Delta)$ to the first equation to the condensed system
\beq\label{B-FD 2}
(1-d\mu\Delta)\p_t\vv V+M(\vv V,D)\vv V=\vv 0,
\eeq
with $M(\vv V,D)$ defined by
\beq\label{M 2}\begin{aligned}
\left(
\begin{array}{ccc}
-\f{\epsilon}{\gamma}(1-d\mu\Delta)(\vv v\cdot\na)&\f{1}{\gamma}(1-d\mu\Delta)\bigl((A(D)-\epsilon\z)\p_1\bigr)&\f{1}{\gamma}(1-d\mu\Delta)\bigl((A(D)-\epsilon\z)\p_2\bigr)\\
(1-\gamma)(1+c\mu\Delta)\p_1&-\f{\epsilon}{\gamma}v^1\p_1&-\f{\epsilon}{\gamma}v^2\p_1\\
(1-\gamma)(1+c\mu\Delta)\p_2&-\f{\epsilon}{\gamma}v^1\p_2&-\f{\epsilon}{\gamma}v^2\p_2
\end{array}\right).\end{aligned}\eeq
The symmetrizer $S_{\vv V}(D)$ of $M(\vv V,D)$  is defined by
\beq\label{S for b=0}\begin{aligned}
\gamma(1-\gamma)\left(
\begin{array}{ccc}
\gamma(1-\gamma)(1+c\mu\Delta)^2&-\epsilon v^1(1+c\mu\Delta)&-\epsilon v^2(1+c\mu\Delta)\\
-\epsilon v^1(1+c\mu\Delta)&(1+c\mu\Delta)[(A(D)-\epsilon\z)(1-d\mu\Delta)]&0\\
-\epsilon v^2(1+c\mu\Delta)&0&(1+c\mu\Delta)[(A(D)-\epsilon\z)(1-d\mu\Delta)]
\end{array}\right)\\
+d\epsilon^2\mu
\left(
\begin{array}{ccc}
0&0&0\\
0&v^1v^1&v^1v^2\\
0&v^1v^2&v^2v^2
\end{array}\right)\D.
\end{aligned}\eeq

We could also have defined the energy functional associated to \eqref{B-FD 2} as
\beq\label{E 2}
E_s(\vv V)=\bigl((1-d\mu\Delta)\Lam^s\vv V\,|\,S_{\vv V}(D)\Lam^s\vv V\bigr)_2.
\eeq
As for the previous choice, we shall show (see Appendix) that $E_s(\vv V)$ defined in \eqref{E 2} is actually a true energy functional equivalent to some $X_{\mu^k}^s(\R^2)$ norm.
\end{remark}

\subsection{Technical lemmas}
We complete this section by recalling some useful well-known results. Firstly, we recall the tame
product estimates in Sobolev spaces: if $t_0>\f{n}{2}\ (n=1,2)$ and $s\geq0$, one has (see \cite{Taylor} Section 3.5)
\beq\label{product}
\|fg\|_{H^s}\lesssim\|f\|_{H^{t_0}}\|g\|_{H^s}+\langle\|f\|_{H^s}\|g\|_{H^{t_0}}\rangle_{s>t_0},
\quad\forall f,g\in H^s\cap H^{t_0}(\R^n).
\eeq

The following interpolation inequality will be also used frequently
\beq\label{interpolation}
\mu^{\f{\theta}{2}}\|f\|_{H^{s+\theta}}\lesssim\|f\|_{H^s}^{1-\f{\theta}{k}}\bigl(\mu^{\f{k}{2}}\|f\|_{H^{s+k}}\bigr)^{\f{\theta}{k}}
\lesssim\|f\|_{X^s_{\mu^k}},
\eeq
where $0<\theta<k$ and $s\geq 0$.

We now present commutator estimates (see Theorems 3 and 6 in \cite{Lannes1}).
\begin{lemma}\label{Commutator lemma 1}
Let $t_0>\f{n}{2}$, $-t_0<r\leq t_0+1$. Then for all $s\geq 0$, $f\in H^{t_0+1}\cap H^{s+r}(\R^n)$ and $u\in H^{s+r-1}(\R^n)$, there holds:
\beq\label{commutator 1}
\|[\Lam^s, f]u\|_{H^r}\lesssim\|\na f\|_{H^{t_0}}\|u\|_{H^{s+r-1}}+\langle\|\na f\|_{H^{s+r-1}}\|u\|_{H^{t_0}}\rangle_{s>t_0+1-r}.
\eeq
\end{lemma}

\medskip

Concerning the Fourier multiplier $g(D)$ for $b,d>0$, $b\neq d$, we have the following lemma (see Lemma 2.3 in \cite{SX}).
\begin{lemma}\label{lemma for g}
Let $b, d>0$ and $b\neq d$, $s\in\R$, $\theta\geq0$. Then
\begin{itemize}
\item[(i)]for all $f\in H^{s}(\R^n)$, there hold
\beq\label{g 1}
\min\{1,(\f{b}{d})^\theta\}\|f\|_{H^s}\leq\|g(D)^\theta f\|_{H^s}\leq\max\{1,(\f{b}{d})^\theta\}\|f\|_{H^s},
\eeq
\beq\label{g 2}
\|(g(D)-1)f\|_{H^s}\leq\f{|b-d|}{d}\|f\|_{H^s};
\eeq
\item[(ii)] let $t_0>\f{n}{2}$, $-t_0<r\leq t_0+1$, for all $f\in H^{t_0+1}(\R^n)$ and $u\in H^{r-1}(\R^n)$, there holds
\beq\label{commutator of g}
\|[g(D)^\theta, f]u\|_{H^r}\leq C\|f\|_{H^{t_0+1}}\|u\|_{H^{r-1}},
\eeq
\end{itemize}
where $C$ is a constant independent of $\mu$.
\end{lemma}

 We now state a useful lemma for the Fourier multiplier $\sigma(D)$.
\begin{lemma}\label{lemma for sigma}
Let $\theta\geq0$. We have
\begin{itemize}
\item[(i)]for all $f\in H^{s}(\R^n)$, there holds
\beq\label{sigma 1}
\mu_2^\theta\||D|^\theta f\|_{L^2}^2\leq\|\sigma(D)^\theta f\|_{L^2}^2\lesssim\|f\|_{H^\theta}^2,
\eeq
\item[(ii)] let $t_0>\f{n}{2}$, $-t_0<r+\theta\leq t_0+1$, for all $f\in H^{t_0+1}(\R^n)$ and $u\in H^{r-1}(\R^n)$, there holds
\beq\label{commutator of sigma}
\|[\sigma(D)^\theta, f]u\|_{H^r}\lesssim\|f\|_{H^{t_0+1}}\|u\|_{H^{r+\theta-1}}.
\eeq
\item[(iii)] let $-t_0<r\leq t_0$, $\theta=\f12,\,1$, we have
\beq\label{commutator of sigma and g}
\|[\sigma(D)^\theta g(D)^{\f12}, f]u\|_{H^r}\lesssim\|f\|_{H^{t_0+1}}\|u\|_{H^{r+\theta-1}}.
\eeq
\end{itemize}
\end{lemma}
\begin{proof}
(i). Recalling the definition of $\sigma(D)$ in \eqref{notation}, by Plancherel theorem, using the fact that $\coth (s)\geq1$, we have
\beno\begin{aligned}
&\|\sigma(D)^\theta f\|_{L^2}^2=(2\pi)^{-n}\int_{\R^n}\bigl(\sqrt{\mu_2}|\xi|\coth(\sqrt{\mu_2}|\xi|)\bigr)^{2\theta}|\hat{f}(\xi)|^2 d\xi\geq(2\pi)^{-n}\int_{\R^n}\bigl(\sqrt{\mu_2}|\xi|\bigr)^{2\theta}|\hat{f}(\xi)|^2 d\xi,
\end{aligned}\eeno
which implies
\beq\label{sigma 1 a}\begin{aligned}
&\|\sigma(D)^\theta f\|_{L^2}^2\geq\mu_2^\theta\||D|^\theta f\|_{L^2}^2.
\end{aligned}\eeq
This is the first part of \eqref{sigma 1}.

Since $\lim_{s\rightarrow 0} s\coth(s)=1$ and $\lim_{s\rightarrow +\infty} \coth(s)=1$, it is easy to get the second part of \eqref{sigma 1}.

\medskip

(ii). Recalling the Definition 9 of \cite{Lannes1}, one could check that $\sigma(\xi)^{\theta}$ is a pseudo-differential operator of order $\theta$.
Indeed, on one hand, for $|\xi|\leq1$, we have
\beno
\sigma(\xi)^{\theta}=\Bigl(\sqrt{\mu_2}|\xi|+\f{2\sqrt{\mu_2}|\xi|}{e^{2\sqrt{\mu_2}|\xi|}-1}\Bigr)^\theta\leq(1+\sqrt{\mu_2})^\theta,
\eeno
which gives rise to
\beq\label{sigma 1 b}
\sup_{|\xi|\leq1}|\sigma(\xi)^\theta|\leq(1+\sqrt{\mu_2})^\theta.
\eeq
On the other hand, for $|\xi|\geq\f14$, $\mu_2\sim1$ and $\beta\in\Z_{\geq0}^n$, it is easy to check
\beq\label{sigma 1 c}
\sup_{|\beta|\leq 2+[\f{n}{2}]+n}\sup_{|\xi|\geq\f14}\langle\xi\rangle^{|\beta|-\theta}|\p_\xi^\beta\sigma(\xi)^\theta|\lesssim 1.
\eeq
Due to \eqref{sigma 1 b} and \eqref{sigma 1 c}, we have $\sigma(D)^\theta\in\Gamma^\theta_\infty$(see Definition 9 of \cite{Lannes1}). Then Theorems 3 and 6 yield \eqref{commutator of sigma}.

\medskip

(iii) Since
\beno
[\sigma(D)^\theta g(D)^{\f12}, f]u=\sigma(D)^\theta \bigl([g(D)^{\f12}, f]u\bigr)+[\sigma(D)^\theta, f]g(D)^{\f12}u,
\eeno
using \eqref{g 1}, \eqref{commutator of g}, \eqref{sigma 1} and \eqref{commutator of sigma}, we have
\beno\begin{aligned}
&\|[\sigma(D)^\theta g(D)^{\f12}, f]u\|_{H^r}\lesssim\|[g(D)^{\f12}, f]u\|_{H^{r+\theta}}+\|f\|_{H^{t_0+1}}\|g(D)^{\f12}u\|_{H^{r+\theta-1}}\lesssim\|f\|_{H^{t_0+1}}\|u\|_{H^{r+\theta-1}}.
\end{aligned}\eeno
This is \eqref{commutator of sigma and g}. We complete the proof of Lemma.
\end{proof}

\setcounter{equation}{0}
\section{Long time existence for \eqref{eqB-FD}-\eqref{initial data}}
The goal of this section is to prove Theorem \ref{Long time existence for B-FD} that is the long time existence of solutions of  \eqref{eqB-FD}-\eqref{initial data}.  The proof follows the same approach used in \cite{SX} that is to derive  energy estimates on suitable symmetrizable linearized system and then use an iterative scheme.

\subsection{Proof of Theorem \ref{Long time existence for B-FD}}
The proof of Theorem \ref{Long time existence for B-FD} relies heavily on the {\it a priori} energy estimates for \eqref{eqB-FD}. To do so,
 we assume
\beq\label{ansatz for amplitude}
1-\epsilon\|\z(t)\|_{L^\infty}\geq\f{H}{2},\quad \sqrt{\epsilon}\|(\z(t),\,\vv v(t))\|_{W^{1,\infty}}\leq 1\quad\hbox{for any}\quad t\in(0,t^*),
\eeq
where $t^*$  will be taken at the end of the proof.
Then we have the following {\it a priori} energy estimates.
\begin{proposition}\label{prop for energy estimate}
Let $s\geq t_0+2$ and $t_0>\f{n}{2}$ with $n=1,2$. Assume that $(\z,\,\vv v)$ are smooth solutions to \eqref{eqB-FD}.
Then under the assumption \eqref{ansatz for amplitude}, there exist small constants $\wt\epsilon_1>0$ and  $\wt\mu>0$ such that for all $\epsilon\leq\wt\epsilon_1$, $\mu\leq\wt\mu$,
\beq\label{priori energy estimate}
\f{d}{dt}E_s(\vv V)\lesssim\epsilon\bigl(1+\epsilon^2 E_s(\vv V)\bigr)\bigl(E_s(\vv V)\bigr)^{\f32},
\eeq
where $E_s(\vv V)$ is defined in \eqref{E 1} or \eqref{E 2}.
\end{proposition}

\begin{remark}
Under the assumption \ref{ansatz for amplitude}, taking $\wt\epsilon>0$ and  $\wt\mu>0$ sufficiently small, there will hold for any $\epsilon\leq\wt\epsilon$, $\mu\leq\wt\mu$
\beq\label{equivalent energy}
E_s(\vv V)\sim\mathcal{E}_s{(t)}\eqdefa\|\z(t)\|_{X^s_{\mu^k}}^2+\|\vv v(t)\|_{X^s_{\mu^{k'}}}^2,
\eeq
where $(k,k')$ is defined in Definition \ref{def2}. We shall use \eqref{equivalent energy} to derive \eqref{priori energy estimate}. The proof of \eqref{equivalent energy} will be postponed to the Appendix for only three typical cases in two dimensional space.
\end{remark}

\begin{proof}[Proof of Theorem \ref{Long time existence for B-FD}]
Assume that
\beq\label{ansatz for energy a}
E_s(\vv V)\leq 16E_s(\vv V_0),\quad\hbox{for any } t\in[0,t^*],
\eeq
where $t^*=\f{T}{\epsilon}$ will be determined later on.
Taking $\wt{\epsilon}_2=\f{1}{4(\mathcal{E}_s(0))^{\f12}}$,
due to \eqref{priori energy estimate}, \eqref{equivalent energy} and \eqref{ansatz for energy a}, for any $\epsilon\leq\wt{\epsilon_2}$,  there exists a constant $C_1>0$ such that
\beno
\f{d}{dt}\bigl(E_s(\vv V)\bigr)^{\f12}\leq C_1\epsilon E_s(\vv V),
\eeno
which gives rise to
\beq\label{est 1}
\bigl(E_s(\vv V)\bigr)^{\f12}\leq\f{\bigl(E_s(\vv V_0)\bigr)^{\f12}}{1-C_1t\epsilon \bigl(E_s(\vv V_0)\bigr)^{\f12}}\leq 2\bigl(E_s(\vv V_0)\bigr)^{\f12},
\eeq
for any $t\leq\f{\wt{T}}{\epsilon}$ with $\wt{T}=\f{1}{2C_1(E_s(\vv V_0))^{\f12}}$. On the other hand, \eqref{equivalent energy} implies there exists a constant $C_2>0$ such that $(E_s(\vv V_0))^{\f12}\leq C_2(\mathcal{E}_s(0))^{\f12}$. Taking
\beno
T=\f{1}{2C_1C_2(\mathcal{E}_s(0))^{\f12}}\leq\wt{T},\quad t^*=T/\epsilon,
\eeno
we have that \eqref{est 1} holds for any $t\leq T/\epsilon$ which improves the ansatz \eqref{ansatz for energy a}. Moreover, using \eqref{equivalent energy} again, we deduce from \eqref{est 1} that for some $C_3>0$,
\beq\label{energy estimate}
\sup_{(0,T/\epsilon)}\mathcal{E}_s(t)\leq C_3\mathcal{E}_s(0).
\eeq

By virtue of Sobolev inequality and \eqref{energy estimate}, noticing that $s\geq t_0+2>3$, there exists a constant $C_4>0$ such that
\beno
\|(\z(t),\,\vv v(t))\|_{W^{1,\infty}}\leq C_4\|(\z(t),\,\vv v(t))\|_{H^s}\leq C_4C_3^{\f12}\bigl(\mathcal{E}_s(0)\bigr)^{\f12}.
\eeno
Taking $\wt{\epsilon}_3=\min\{\f{1-H}{C_4C_3^{\f12}\bigl(\mathcal{E}_s(0)\bigr)^{\f12}},\,\f{1}{4C_4^2C_3\mathcal{E}_s(0)}\}$, we have for any $\epsilon\leq\min\{\wt\epsilon_2,\wt\epsilon_3\}$,
\beq\label{est 2}
1-\epsilon\|\z\|_{L^\infty}\geq H>\f{H}{2},\quad\sqrt\epsilon\|(\z(t),\,\vv v(t))\|_{W^{1,\infty}}\leq\f12,
\eeq
which improves the ansatz \eqref{ansatz for amplitude}. Then taking $\wt\epsilon=\min\{\wt\epsilon_1,\wt\epsilon_2,\wt\epsilon_3\}$, we have for any $\epsilon\leq\wt\epsilon$ and $\mu\leq\wt\mu$, energy estimate \eqref{energy estimate} holds for any $t\in[0,T/\epsilon]$. Thus, \eqref{long time estimate} is proved.

The {\it existence} and {\it uniqueness} of the solution can be verified by standard mollification method and the Cauchy-Lipschitz theorem. One could refer to \cite{SX}. Now, we complete the proof of Theorem \ref{Long time existence for B-FD}.
\end{proof}

The rest of this section is devoted to prove Proposition \ref{prop for energy estimate}. We only sketch the proof of three typical cases in two dimensional space, since the others could be treated in a similar way.

\subsection{A priori estimates for the  "general case": $b\neq d,\, b>0,\, d>0, \, a\leq 0,\, c<0$.}
In this case, one could check that
\beq\label{equivalent 1}
E_s(\vv V)\sim\mathcal{E}_s{(t)}\eqdefa\|\z(t)\|_{X^s_{\mu^3}}^2+\|\vv v(t)\|_{X^s_{\mu^{3}}}^2
\eeq
for any $\epsilon\leq\wt\epsilon_1$ and $\mu\leq\wt\mu$ with $\wt\epsilon_1$ and $\wt\mu$ being sufficiently small. We postpone the proof of \eqref{equivalent 1} to Appendix.

A direct energy estimate shows that
\beq\label{est 3}\begin{aligned}
&\f{d}{dt}E_s(\vv V)=\bigl((1-b\mu\Delta)\Lam^s\p_t\vv V\,|\,(S_{\vv V}(D)+S_{\vv V}(D)^*)\Lam^s\vv V\bigr)_2\\
&\quad-b\mu([S_{\vv V}(D)^*, \Delta]\Lam^s\vv V\,|\,\Lam^s\p_t\vv V)_2
+\bigl((1-b\mu\Delta)\Lam^s\vv V\,|\,\p_tS_{\vv V}(D)\Lam^s\vv V\bigr)_2\\
&\eqdefa I+II+III,
\end{aligned}\eeq
where $S_{\vv V}(D)^*$ is the adjoint operator of $S_{\vv V}(D)$.

{\bf Step 1. Estimate on $I$. } Using \eqref{B-FD 1}, we have
\beq\label{est 4}\begin{aligned}
&I=-\bigl([\Lam^s,M(\vv V,D)]\vv V\,|\,(S_{\vv V}(D)+S_{\vv V}(D)^*)\Lam^s\vv V\bigr)_2\\
&\quad
-\bigl((S_{\vv V}(D)+S_{\vv V}(D)^*)\bigl(M(\vv V,D)\Lam^s\vv V\bigr)\,|\,\Lam^s\vv V\bigr)_2\eqdefa I_1+I_2.
\end{aligned}\eeq

{\it Step 1.1. Estimate on $I_1$.} Using \eqref{M 1} and \eqref{S for b neq d}, a direct calculation yields
\beq\label{est a}\begin{aligned}
&\quad\bigl([\Lam^s,M(\vv V,D)]\vv V\,|\,S_{\vv V}(D)\Lam^s\vv V\bigr)_2\\
&=-\epsilon(1-\gamma)^2\gamma\bigl(g(D)\bigl([\Lam^s,\vv v]\cdot\na\z+[\Lam^s,\z]\na\cdot\vv v\bigr)\,|\,(1+c\mu\D)^2\Lam^s\z\bigr)_2\\
&\quad+\epsilon^2(1-\gamma)\bigl(g(D)\bigl([\Lam^s,\vv v]\cdot\na\z+[\Lam^s,\z]\na\cdot\vv v\bigr)\,|\,\vv v\cdot(1+c\mu\D)\Lam^s\vv v\bigr)_2\\
&\quad+\epsilon^2(1-\gamma)\sum_{j=1,2}\bigl(g(D)\bigl([\Lam^s,\vv v]\cdot\p_j\vv v\bigr)\,|\, g(D)\bigl(v^j(1+c\mu\D)\Lam^s\z\bigr)+\z(1+c\mu\D)\Lam^sv^j\bigr)_2\\
&\quad
-\epsilon(1-\gamma)\sum_{j=1,2}\bigl(g(D)^{\f12}(1+c\mu\D)\bigl([\Lam^s,\vv v]\cdot\p_j\vv v\bigr)\,|\,g(D)^{\f12}A(D)\Lam^sv^j\bigr)_2\\
&\quad-\f{\epsilon^3}{\gamma}\sum_{j=1,2}\bigl(g(D)\bigl([\Lam^s,\vv v]\cdot\p_j\vv v\bigr)\,|\,v^j\vv v\cdot(g(D)-1)\Lam^s\vv v\bigr)_2\eqdefa I_{11}+I_{12}+I_{13}+I_{14}+I_{15}.
\end{aligned}\eeq

For $I_{11}$, integration by parts yields
\beno\begin{aligned}
&|I_{11}|\lesssim\epsilon\|g(D)\bigl([\Lam^s,\vv v]\cdot\na\z+[\Lam^s,\z]\na\cdot\vv v\bigr)\|_{L^2}\|(1+c\mu\D)\Lam^s\z\|_{L^2}\\
&\quad
+|c|\epsilon\mu\|g(D)\na\bigl([\Lam^s,\vv v]\cdot\na\z+[\Lam^s,\z]\na\cdot\vv v\bigr)\|_{L^2}\|(1+c\mu\D)\na\Lam^s\z\|_{L^2},
\end{aligned}\eeno
By virtue of \eqref{g 1} and \eqref{commutator 1}, noticing that $s\geq t_0+2>3$, we have
\beno\begin{aligned}
&\|g(D)\bigl([\Lam^s,\vv v]\cdot\na\z\bigr)\|_{L^2}\lesssim\|[\Lam^s,\vv v]\cdot\na\z\|_{L^2}\lesssim\|\vv v\|_{H^{t_0+1}}\|\z\|_{H^s}+\|\vv v\|_{H^s}\|\z\|_{H^{t_0+1}}\lesssim\|\z\|_{H^s}\|\vv v\|_{H^s},\\
&\|g(D)\na([\Lam^s,\vv v]\cdot\na\z)\|_{L^2}\lesssim\|\vv v\|_{H^{t_0+1}}\|\z\|_{H^{s+1}}+\|\vv v\|_{H^{s+1}}\|\z\|_{H^{t_0+1}}\lesssim\|\vv v\|_{H^{s}}\|\z\|_{H^{s+1}}+\|\vv v\|_{H^{s+1}}\|\z\|_{H^{s}}.
\end{aligned}\eeno
Similar estimates hold for $\|g(D)\bigl([\Lam^s,\z]\na\cdot\vv v\bigr)\|_{L^2}$ and $\|g(D)\na\bigl([\Lam^s,\z]\na\cdot\vv v\bigr)\|_{L^2}$. Since
\beno\begin{aligned}
&\|(1+c\mu\D)\Lam^s\z\|_{L^2}\lesssim\|\z\|_{H^{s}}+\mu\|\z\|_{H^{s+2}},\quad\|\na(1+c\mu\D)\Lam^s\z\|_{L^2}\lesssim\|\z\|_{H^{s+1}}+\mu\|\z\|_{H^{s+3}},
\end{aligned}\eeno
we have
\beno\begin{aligned}
&|I_{11}|\lesssim\epsilon\|\z\|_{H^s}\|\vv v\|_{H^s}\bigl(\|\z\|_{H^s}+\mu\|\z\|_{H^{s+2}}\bigr)\\
&\quad
+\epsilon\bigl(\|\vv v\|_{H^{s}}\cdot\mu^{\f12}\|\z\|_{H^{s+1}}+\mu^{\f12}\|\vv v\|_{H^{s+1}}\cdot\|\z\|_{H^{s}}\bigr)
\bigl(\mu^{\f12}\|\z\|_{H^{s+1}}+\mu^{\f32}\|\z\|_{H^{s+3}}\bigr)
\end{aligned}\eeno
which along with \eqref{interpolation} implies
\beq\label{est 5}
|I_{11}|\lesssim\epsilon\|\vv v\|_{X^s_{\mu^3}}\|\z\|_{X^s_{\mu^3}}^2.
\eeq

For $I_{14}$,  we first have
\beno\begin{aligned}
&|I_{14}|\lesssim\epsilon\sum_{j=1,2}\bigl(\|g(D)^{\f12}(1+c\mu\D)\bigl([\Lam^s,\vv v]\cdot\p_j\vv v\bigr)\|_{L^2}\|g(D)^{\f12}A(D)\Lam^sv^j\|_{L^2}.
\end{aligned}\eeno
Recalling that
$
A(D)=1+a\mu\D+\f{1}{\gamma}\sqrt{\f{\mu}{\mu_2}}\sigma(D)
+\f{1}{\gamma^2}\f{\mu}{\mu_2}\sigma(D)^2,
$
using \eqref{sigma 1} and \eqref{interpolation}, we have
\beq\label{est 6}
\|A(D)f\|_{L^2}\lesssim\|f\|_{L^2}+\mu^{\f12}\|\na f\|_{L^2}+\mu\|\na^2 f\|_{L^2}\lesssim\|f\|_{X^s_{\mu^2}}.
\eeq

Following a similar derivation as \eqref{est 5}, using \eqref{commutator 1}, \eqref{g 1},  \eqref{est 6} and \eqref{interpolation}, we arrive at
\beq\label{est 7}
|I_{14}|\lesssim\epsilon\|\vv v\|_{X^s_{\mu^3}}^3.
\eeq

Similar estimates as \eqref{est 5} and \eqref{est 7} hold for $I_{12}$, $I_{13}$ and $I_{15}$. Then we get
\beno
|\bigl([\Lam^s,M(\vv V,D)]\vv V\,|\,S_{\vv V}(D)\Lam^s\vv V\bigr)_2|\lesssim\epsilon(1+\epsilon\|\vv v\|_{L^\infty})^2\bigl(\|\z\|_{X^s_{\mu^3}}+\|\vv v\|_{X^s_{\mu^3}}\bigr)^3.
\eeno
The same estimate holds for $\bigl([\Lam^s,M(\vv V,D)]\vv V\,|\,S_{\vv V}(D)^*\Lam^s\vv V\bigr)_2$. Using \eqref{equivalent 1}, we obtain
\beq\label{estimate for I1}
|I_1|\lesssim\epsilon\bigl(1+\epsilon^2 E_s(\vv V)\bigr)\bigl(E_s(\vv V)\bigr)^{\f32}.
\eeq

\smallskip

{\it Step 1.2. Estimate on $I_2$.} In order to estimate $I_2$, we first calculate $S_{\vv V}(D)M(\vv V,D)\eqdefa\mathcal{A}_{\vv V}(D)\eqdefa(a_{ij})_{i,j=1,2,3}$ as follows:
\beq\label{expression of A}\begin{aligned}
&a_{11}=-\epsilon\gamma(1-\gamma)^2[(1+c\mu\D)^2g(D)(\vv v\cdot\na)+g(D)\bigl(\vv v\cdot\na(1+c\mu\D)^2g(D)\bigr)]\\
&\qquad\eqdefa
-\epsilon\gamma(1-\gamma)^2(a_{111}+a_{112}),\\
&a_{12}=\gamma(1-\gamma)^2(1+c\mu\D)^2g(D)\bigl((A(D)-\epsilon\z)\p_1\bigr)+\epsilon^2(1-\gamma)g(D)\bigl(\vv v\cdot(1+c\mu\D)g(D)(v^1\na)\bigr)\\
&\qquad
\eqdefa a_{121}+a_{122},\\
&a_{13}=\gamma(1-\gamma)^2(1+c\mu\D)^2g(D)\bigl((A(D)-\epsilon\z)\p_2\bigr)+\epsilon^2(1-\gamma)g(D)\bigl(\vv v\cdot(1+c\mu\D)g(D)(v^2\na)\bigr)\\
&a_{21}=\gamma(1-\gamma)^2(A(D)-\epsilon\z)(1+c\mu\D)^2g(D)\p_1+\epsilon^2(1-\gamma)g(D)\bigl(v^1(1+c\mu\D)(\vv v\cdot\na)\bigr)\\
&\qquad+\epsilon^2(1-\gamma)v^1\vv v\cdot\na(g(D)-1)g(D)(1+c\mu\D)\eqdefa a_{211}+a_{212}+a_{213},\\
&a_{22}=-\epsilon(1-\gamma)g(D)\bigl[v^1(1+c\mu\D)\bigl((A(D)-\epsilon\z)\p_1\bigr)\bigr]
-\epsilon(1-\gamma)(A(D)-\epsilon\z)(1+c\mu\D)g(D)(v^1\p_1)\\
&\qquad-\f{\epsilon^3}{\gamma}v^1\vv v\cdot(g(D)-1)g(D)(v^1\na)
\eqdefa a_{221}+a_{222}+a_{223},\\
&a_{23}=-\epsilon(1-\gamma)g(D)\bigl[v^1(1+c\mu\D)\bigl((A(D)-\epsilon\z)\p_2\bigr)\bigr]
-\epsilon(1-\gamma)(A(D)-\epsilon\z)(1+c\mu\D)g(D)(v^2\p_1)\\
&\qquad-\f{\epsilon^3}{\gamma}v^1\vv v\cdot(g(D)-1)g(D)(v^2\na)
\eqdefa a_{231}+a_{232}+a_{233}\\
&a_{31}=\gamma(1-\gamma)^2(A(D)-\epsilon\z)(1+c\mu\D)^2g(D)\p_2+\epsilon^2(1-\gamma)g(D)\bigl(v^2(1+c\mu\D)(\vv v\cdot\na)\bigr)\\
&\qquad+\epsilon^2(1-\gamma)v^2\vv v\cdot\na(g(D)-1)g(D)(1+c\mu\D),\\
&a_{32}=
-\epsilon(1-\gamma)(A(D)-\epsilon\z)(1+c\mu\D)g(D)(v^1\p_2)-\epsilon(1-\gamma)g(D)\bigl[v^2(1+c\mu\D)\bigl((A(D)-\epsilon\z)\p_1\bigr)\bigr]\\
&\qquad-\f{\epsilon^3}{\gamma}v^2\vv v\cdot(g(D)-1)g(D)(v^1\na)
\eqdefa a_{321}+a_{322}+a_{323},\\
&a_{33}=-\epsilon(1-\gamma)g(D)\bigl[v^2(1+c\mu\D)\bigl((A(D)-\epsilon\z)\p_2\bigr)\bigr]
-\epsilon(1-\gamma)(A(D)-\epsilon\z)(1+c\mu\D)g(D)(v^2\p_2)\\
&\qquad-\f{\epsilon^3}{\gamma}v^2\vv v\cdot(g(D)-1)g(D)(v^2\na).
\end{aligned}\eeq

The expression of $\mathcal{A}_{\vv V}(D)$ shows that the principal part of $i\mathcal{A}_{\vv V}(D)$ is symmetric.
Now, we estimate $\bigl(S_{\vv V}(D)M(\vv V,D)\Lam^s\vv V\,|\,\Lam^s\vv V\bigr)_2=\bigl(\mathcal{A}_{\vv V}(D)\Lam^s\vv V\,|\,\Lam^s\vv V\bigr)_2$ term by term.

\smallskip

{\it For $a_{11}$,} we have
\beno
\bigl(a_{11}\Lam^s\z\,|\,\Lam^s\z\bigr)_2
=-\epsilon\gamma(1-\gamma)^2\{\bigl(a_{111}\Lam^s\z\,|\,\Lam^s\z\bigr)_2+\bigl(a_{112}\Lam^s\z\,|\,\Lam^s\z\bigr)_2\}.
\eeno

Using the expression of $a_{111}$, integrating by parts, we have
\beq\label{est 8}\begin{aligned}
&\bigl(a_{111}\Lam^s\z\,|\,\Lam^s\z\bigr)_2=\bigl(g(D)^{\f12}\bigl([c\mu\D,\vv v]\cdot\na\Lam^s\z\bigr)\,|\,g(D)^{\f12}(1+c\mu\D)\Lam^s\z\bigr)_2\\
&\quad+
\bigl([g(D)^{\f12},\vv v]\cdot\na(1+c\mu\D)\Lam^s\z\,|\,g(D)^{\f12}(1+c\mu\D)\Lam^s\z\bigr)_2\\
&\quad+\bigl(\vv v\cdot\na g(D)^{\f12}(1+c\mu\D)\Lam^s\z\,|\,g(D)^{\f12}(1+c\mu\D)\Lam^s\z\bigr)_2.
\end{aligned}\eeq

Integration by parts yields that the last term in \eqref{est 8} equals
\beno\begin{aligned}
&-\f12\bigl(\na\cdot\vv v g(D)^{\f12}(1+c\mu\D)\Lam^s\z\,|\,g(D)^{\f12}(1+c\mu\D)\Lam^s\z\bigr)_2
\end{aligned}\eeno
which along with \eqref{est 8} and \eqref{g 1} implies
\beno\begin{aligned}
&|\bigl(a_{111}\Lam^s\z\,|\,\Lam^s\z\bigr)_2|\lesssim\mu\|[\D,\vv v]\cdot\na\Lam^s\z\|_{L^2}\|\z\|_{X^s_{\mu^2}}
+\|[g(D)^{\f12},\vv v]\cdot\na(1+c\mu\D)\Lam^s\z\|_{L^2}\|\z\|_{X^{s}_{\mu^2}}
+\|\na\vv v\|_{L^\infty}\|\z\|_{X^{s}_{\mu^2}}^2.
\end{aligned}\eeno
Thanks to \eqref{commutator of g}, we have
\beno\begin{aligned}
&\|[g(D)^{\f12},\vv v]\cdot\na(1+c\mu\D)\Lam^s\z\|_{L^2}\lesssim\|\vv v\|_{H^{t_0+1}}\|\z\|_{X^s_{\mu^2}}.
\end{aligned}\eeno

Since
\beno
\mu\|[\D,\vv v]\cdot\na\Lam^s\z\|_{L^2}\lesssim\mu\|\vv v\|_{H^{t_0+2}}\|\z\|_{H^{s+1}}+\mu\|\vv v\|_{H^{t_0+1}}\|\z\|_{H^{s+2}}
\eeno
using \eqref{interpolation}, we get
\beno
|\bigl(a_{111}\Lam^s\z\,|\,\Lam^s\z\bigr)_2|\lesssim\|\vv v\|_{H^{t_0+2}}\|\z\|_{X^s_{\mu^3}}^2
\lesssim\|\vv v\|_{H^s}\|\z\|_{X^s_{\mu^3}}^2.
\eeno
The same estimate holds for $\bigl(a_{112}\Lam^s\z\,|\,\Lam^s\z\bigr)_2$. Then we obtain
\beq\label{est 9}
|\bigl(a_{11}\Lam^s\z\,|\,\Lam^s\z\bigr)_2|\lesssim\epsilon\|\vv v\|_{H^{t_0+2}}\|\z\|_{X^s_{\mu^3}}^2
\lesssim\epsilon\|\vv v\|_{H^s}\|\z\|_{X^s_{\mu^3}}^2.
\eeq

\smallskip

{\it For $a_{12}$ and $a_{21}$,} we have
\beno
\bigl(a_{12}\Lam^sv^1\,|\,\Lam^s\z\bigr)_2+\bigl(a_{21}\Lam^s\z\,|\,\Lam^sv^1\bigr)_2
=\bigl((a_{12}^*+a_{21})\Lam^s\z\,|\,\Lam^s v^1\bigr)_2,
\eeno
where $a_{12}^*$ is  the adjoint operator of $a_{12}$. By the expression of $a_{12}$, we first have
\beno\begin{aligned}
&a_{121}^*=-\gamma(1-\gamma)^2\p_1\bigl[(A(D)-\epsilon\z)(1+c\mu\D)^2g(D)\bigr]\\
&a_{122}^*=-\epsilon^2(1-\gamma)
\na\cdot\bigl[v^1(1+c\mu\D)g(D)(\vv v g(D))\bigr].
\end{aligned}\eeno
Due to the expression of  $a_{21}$, we have
\beno\begin{aligned}
&a_{121}^*+a_{211}=\epsilon\gamma(1-\gamma)^2\p_1\z(1+c\mu\D)^2g(D),
\end{aligned}\eeno
which implies
\beq\label{est 10a}\begin{aligned}
&\bigl((a_{121}^*+a_{211})\Lam^s\z\,|\,\Lam^s v^1\bigr)_2=\epsilon\gamma(1-\gamma)^2\bigl((1+c\mu\D)g(D)\Lam^s\z\,|\,(1+c\mu\D)(\p_1\z\Lam^s v^1)\bigr)_2.
\end{aligned}\eeq
Noticing that $s\geq t_0+2>3$, using \eqref{g 1} and \eqref{interpolation}, we have
\beq\label{est 10}
|\bigl((a_{121}^*+a_{211})\Lam^s\z\,|\,\Lam^s v^1\bigr)_2|\lesssim\epsilon\|\z\|_{X^{t_0+1}_{\mu^2}}\|\z\|_{X^s_{\mu^2}}\|\vv v\|_{X^s_{\mu^2}}\lesssim\epsilon\|\z\|_{X^s_{\mu^3}}^2\|\vv v\|_{X^s_{\mu^3}}.
\eeq

Since
\beno\begin{aligned}
a_{122}^*&=-
\epsilon^2(1-\gamma)\na\cdot\bigl[v^1(1+c\mu\D)g(D)(\vv v )\bigr]-
\epsilon^2(1-\gamma)\na\cdot\bigl[v^1(1+c\mu\D)g(D)\bigl(\vv v(g(D)-1)\bigr)\bigr]\\
&\eqdefa a_{122,1}^*+a_{122,2}^*,
\end{aligned}\eeno
we have
\beq\label{est 22a}\begin{aligned}
&\f{1}{\epsilon^2(1-\gamma)}(a_{122,1}^*+a_{212})=-
\na v^1\cdot(1+c\mu\D)g(D)(\vv v\cdot )-
v^1(1+c\mu\D)g(D)(\na\cdot\vv v \cdot)\\
&\quad+[g(D),v^1]\bigl((1+c\mu\D)(\vv v\cdot\na)\bigr),
\end{aligned}\eeq
and
\beq\label{est 22b}\begin{aligned}
&\f{1}{\epsilon^2(1-\gamma)}(a_{122,2}^*+a_{213})
=-\na v^1\cdot(1+c\mu\D)g(D)\bigl(\vv v(g(D)-1)\bigr)\\
&\quad-v^1(1+c\mu\D)g(D)\bigl(\na\cdot\vv v(g(D)-1)\bigr)
-v^1[(1+c\mu\D)g(D),\vv v]\cdot\na(g(D)-1),
\end{aligned}\eeq
which long with \eqref{g 1}, \eqref{g 2}, \eqref{commutator of g} and \eqref{interpolation} implies
\beno\begin{aligned}
&\|(a_{122,1}^*+a_{212})\Lam^s\z\|_{L^2}
+\|(a_{122,2}^*+a_{213})\Lam^s\z\|_{L^2}\lesssim\epsilon^2\|v^1\|_{H^{t_0+1}}\|\vv v\|_{X^{t_0+1}_{\mu^2}}\|\z\|_{X^s_{\mu^2}},
\end{aligned}\eeno
where we used the formula
\beno
[(1+c\mu\D)g(D),\vv v]=(1+c\mu\D)\bigl([g(D),\vv v]\bigr)+c\mu[\D,\vv v]g(D).
\eeno
Since $s\geq t_0+2>3$, using \eqref{interpolation} again, we have
\beq\label{est 11}
|\bigl((a_{122}^*+a_{212}+a_{213})\Lam^s\z\,|\,\Lam^s v^1\bigr)_2|\lesssim\epsilon^2\|\vv v\|_{X^{t_0+1}_{\mu^2}}^2\|\z\|_{X^s_{\mu^2}}\|v^1\|_{H^s}\lesssim\epsilon^2\|\z\|_{X^s_{\mu^3}}\|\vv v\|_{X^s_{\mu^3}}^3.
\eeq

Thanks to \eqref{est 10} and \eqref{est 11}, we have
\beq\label{est 12}
|\bigl(a_{12}\Lam^sv^1\,|\,\Lam^s\z\bigr)_2+\bigl(a_{21}\Lam^s\z\,|\,\Lam^sv^1\bigr)_2|
\lesssim\epsilon\bigl(1+\epsilon\|\vv v\|_{X^s_{\mu^3}}\bigr)\|\vv v\|_{X^s_{\mu^3}}\bigl(\|\z\|_{X^s_{\mu^3}}^2+\|\vv v\|_{X^s_{\mu^3}}^2\bigr).
\eeq
The same estimate holds for $\bigl(a_{13}\Lam^sv^2\,|\,\Lam^s\z\bigr)_2+\bigl(a_{31}\Lam^s\z\,|\,\Lam^sv^2\bigr)_2$.

\smallskip

{\it For $a_{22}$,} we first estimate $\bigl(a_{221}\Lam^sv^1\,|\,\Lam^sv^1\bigr)_2$. Using the expression of $a_{221}$, we have
\beno\begin{aligned}
&-\f{1}{1-\gamma}\bigl(a_{221}\Lam^sv^1\,|\,\Lam^sv^1\bigr)_2
=\epsilon\bigl(g(D)\bigl(v^1A(D)\p_1\Lam^sv^1\bigr)\,|\,\Lam^sv^1\bigr)_2
+c\epsilon\mu\bigl(g(D)\bigl(v^1\D A(D)\p_1\Lam^sv^1\bigr)\,|\,\Lam^sv^1\bigr)_2\\
&\quad
-\epsilon^2\bigl(g(D)\bigl(v^1\z\p_1\Lam^sv^1\bigr)\,|\,\Lam^sv^1\bigr)_2
-c\epsilon^2\mu\bigl(g(D)\bigl(v^1\D(\z\p_1\Lam^sv^1)\bigr)\bigr]\,|\,\Lam^sv^1\bigr)_2\eqdefa B_{11}+B_{12}+B_{13}+B_{14}.
\end{aligned}\eeno

For $B_{11}$, using the expression of $A(D)$ in \eqref{notation}, we have
\beno\begin{aligned}
&B_{11}=\epsilon\bigl(g(D)^{\f12}\bigl(v^1\p_1\Lam^sv^1\bigr)\,|\,g(D)^{\f12}\Lam^sv^1\bigr)_2
+a\epsilon\mu\bigl(g(D)^{\f12}\bigl(v^1\D\p_1\Lam^sv^1\bigr)\,|\,g(D)^{\f12}\Lam^sv^1\bigr)_2\\
&\quad
+\f{\epsilon}{\gamma}\sqrt{\f{\mu}{\mu_2}}\bigl(g(D)^{\f12}\bigl(v^1\sigma(D)\p_1\Lam^sv^1\bigr)\,|\,g(D)^{\f12}\Lam^sv^1\bigr)_2
+\f{\epsilon}{\gamma^2}\f{\mu}{\mu_2}\bigl(g(D)^{\f12}\bigl(v^1\sigma(D)^2\p_1\Lam^sv^1\bigr)\,|\,g(D)^{\f12}\Lam^sv^1\bigr)_2\\
&\quad\eqdefa B_{11,1}+B_{11,2}+B_{11,3}+B_{11,4}.
\end{aligned}\eeno

A direct calculation shows that
\beno\begin{aligned}
&\gamma^2\f{\mu_2}{\epsilon\mu}B_{11,4}=\bigl([g(D)^{\f12},v^1]\sigma(D)^2\p_1\Lam^sv^1\,|\,g(D)^{\f12}\Lam^sv^1\bigr)_2\\
&\quad-\bigl([\sigma(D),v^1]\sigma(D)\p_1g(D)^{\f12}\Lam^sv^1\,|\,g(D)^{\f12}\Lam^sv^1\bigr)_2
+\bigl(v^1\p_1\sigma(D)g(D)^{\f12}\Lam^sv^1\,|\,\sigma(D)g(D)^{\f12}\Lam^sv^1\bigr)_2.
\end{aligned}\eeno
Integrating by parts for the last term of $\gamma^2\f{\mu_2}{\epsilon\mu}B_{11,4}$, we see that it equals
\beno
-\f12\bigl(\p_1v^1\sigma(D)g(D)^{\f12}\Lam^sv^1\,|\,\sigma(D)g(D)^{\f12}\Lam^sv^1\bigr)_2.
\eeno
Using \eqref{g 1} and \eqref{sigma 1}, we have
\beq\label{est 13}\begin{aligned}
&|B_{11,4}|\lesssim\epsilon\mu\bigl(\|[g(D)^{\f12},v^1]\sigma(D)^2\p_1\Lam^sv^1\|_{L^2}
+\|[\sigma(D),v^1]\sigma(D)\p_1g(D)^{\f12}\Lam^sv^1\|_{L^2}\bigr)\|v^1\|_{H^s}\\
&\qquad+\epsilon\mu\|\p_1v^1\|_{L^\infty}\|v^1\|_{H^{s+1}}^2.
\end{aligned}\eeq

Thanks to \eqref{commutator of g}, \eqref{commutator of sigma}, \eqref{g 1} and \eqref{sigma 1}, we have
\beno\begin{aligned}
&\|[g(D)^{\f12},v^1]\sigma(D)^2\p_1\Lam^sv^1\|_{L^2}\lesssim\|v^1\|_{H^{t_0+1}}\|\sigma(D)^2\p_1\Lam^sv^1\|_{H^{-1}}
\lesssim\|v^1\|_{H^s}\|v^1\|_{H^{s+2}},\\
&\|[\sigma(D),v^1]\sigma(D)\p_1g(D)^{\f12}\Lam^sv^1\|_{L^2}\lesssim\|v^1\|_{H^{t_0+1}}\|\sigma(D)\p_1g(D)^{\f12}\Lam^sv^1\|_{L^2}
\lesssim\|v^1\|_{H^s}\|v^1\|_{H^{s+2}}
\end{aligned}\eeno
which along with \eqref{interpolation} and \eqref{est 13} imply
\beno\begin{aligned}
&|B_{11,4}|\lesssim\epsilon\mu\|v^1\|_{H^{s+2}}\cdot\|v^1\|_{H^s}^2+\epsilon\mu\|v^1\|_{H^{s+1}}^2\cdot\|v^1\|_{H^s}
\lesssim\epsilon\|v^1\|_{X^{s}_{\mu^3}}^3.
\end{aligned}\eeno
Similar estimates hold for $B_{11,1}$, $B_{11,2}$ and $B_{11,3}$. Then we obtain
\beq\label{est 14}
\begin{aligned}
&|B_{11}|\lesssim\epsilon\|v^1\|_{H^s}\|v^1\|_{X^{s}_{\mu^2}}^2\lesssim\epsilon\|v^1\|_{X^{s}_{\mu^3}}^3.
\end{aligned}
\eeq

Following similar derivation as \eqref{est 14}, we have
\beno
\begin{aligned}
&|B_{12}|\lesssim\epsilon\|v^1\|_{X^{t_0+1}_{\mu^2}}\|v^1\|_{X^{s}_{\mu^3}}^2\lesssim\epsilon\|v^1\|_{X^{s}_{\mu^3}}^3,\\
&|B_{13}|+|B_{14}|\lesssim\epsilon^2\|v^1\|_{H^{t_0+1}}\|\z\|_{H^{t_0+2}}\|v^1\|_{X^{s}_{\mu^2}}^2
\lesssim\epsilon^2\|\z\|_{H^s}\|v^1\|_{H^s}\|v^1\|_{X^{s}_{\mu^3}}^2,
\end{aligned}
\eeno
which along with \eqref{est 14} imply
\beq\label{est 15 a}
|\bigl(a_{221}\Lam^sv^1\,|\,\Lam^s v^1\bigr)_2|
\lesssim\epsilon\bigl(1+\epsilon\|\z\|_{X^s_{\mu^3}}\bigr)\|\vv v\|_{X^s_{\mu^3}}^3.
\eeq
Similarly, we have
\beno\begin{aligned}
&|\bigl(a_{222}\Lam^sv^1\,|\,\Lam^s v^1\bigr)_2|
\lesssim\epsilon\bigl(1+\epsilon\|\z\|_{X^s_{\mu^3}}\bigr)\|\vv v\|_{X^s_{\mu^3}}^3,\\
&|\bigl(a_{223}\Lam^sv^1\,|\,\Lam^s v^1\bigr)_2|
\lesssim\epsilon\bigl(1+\epsilon^2\|\z\|_{X^s_{\mu^3}}^2\bigr)\|\vv v\|_{X^s_{\mu^3}}^3,
\end{aligned}\eeno
which along with \eqref{est 15 a} implies
\beq\label{est 15}
|\bigl(a_{22}\Lam^sv^1\,|\,\Lam^s v^1\bigr)_2|
\lesssim\epsilon\bigl(1+\epsilon\|\z\|_{X^s_{\mu^3}}\bigr)\|\vv v\|_{X^s_{\mu^3}}^3.
\eeq
The same estimate holds for $\bigl(a_{33}\Lam^sv^2\,|\,\Lam^sv^2\bigr)_2$.

\smallskip

{\it For $a_{23}$ and $a_{32}$,} we have
\beno
\bigl(a_{23}\Lam^sv^2\,|\,\Lam^sv^1\bigr)_2+\bigl(a_{32}\Lam^sv^1\,|\,\Lam^sv^2\bigr)_2
=\bigl((a_{23}^*+a_{32})\Lam^sv^1\,|\,\Lam^s v^2\bigr)_2,
\eeno
where $a_{23}^*$ is  the adjoint operator of $a_{23}$. By the expression of $a_{23}=a_{231}+a_{232}+a_{233}$, we first have
\beno
\begin{aligned}
&a_{231}^*=\epsilon(1-\gamma)\p_2\bigl[\bigl(A(D)-\epsilon\z\bigr)(1+c\mu\D)\bigl(v^1g(D)\bigr)\bigr],\\
&a_{232}^*=\epsilon(1-\gamma)\p_1\bigl[v^2g(D)(1+c\mu\D)\bigl((A(D)-\epsilon\z)\cdot\bigr)\bigr],\\
&a_{233}^*=\f{\epsilon^3}{\gamma}\na\cdot\bigl[v^2g(D)(g(D)-1)(\vv vv^1\cdot)\bigr],
\end{aligned}
\eeno
which along with the expression of $a_{32}$ imply
\beq\label{est 23a}\begin{aligned}
&a_{231}^*+a_{321}=-\epsilon^2(1-\gamma)\p_2\z(1+c\mu\D)\bigl(v^1g(D)\bigr)
+\epsilon(1-\gamma)\bigl(A(D)-\epsilon\z\bigr)(1+c\mu\D)\bigl(\p_2v^1g(D)\bigr)\\
&\qquad
-\epsilon(1-\gamma)(A(D)-\epsilon\z)(1+c\mu\D)\bigl([g(D),\,v^2]\p_2\bigr),\\
&a_{232}^*+a_{322}=\epsilon(1-\gamma)\p_1v^2g(D)(1+c\mu\D)\bigl((A(D)-\epsilon\z)\cdot\bigr)
-\epsilon^2(1-\gamma)v^2(1+c\mu\D)g(D)\bigl(\p_1\z\cdot\bigr)\\
&\qquad
-\epsilon(1-\gamma)[g(D),\,v^2](1+c\mu\D)\bigl((A(D)-\epsilon\z)\p_1\bigr),\\
&a_{233}^*+a_{323}=\f{\epsilon^3}{\gamma}\{\na v^2\cdot g(D)(g(D)-1)(\vv vv^1\cdot)+ v^2g(D)(g(D)-1)\bigl(\na\cdot(\vv vv^1)\cdot\bigr)\\
&\qquad
+v^2[g(D)(g(D)-1),\vv v]\cdot (v^1\na)\}.
\end{aligned}\eeq

Thanks to \eqref{notation}, \eqref{g 1}, \eqref{est 6}, \eqref{commutator of g} and \eqref{interpolation}, we have
\beno\begin{aligned}
&|\bigl((a_{232}^*+a_{322})\Lam^s v^1\,|\,\Lam^s v^2\bigr)_2|\lesssim\epsilon\|(A(D)-\epsilon\z)\Lam^s v^1\|_{L^2}\|(1+c\mu\D)(\p_1v^2\Lam^s v^2)\|_{L^2}\\
&\qquad
+\epsilon^2\|v^2(1+c\mu\D)g(D)\bigl(\p_1\z\Lam^s v^1\bigr)\|_{L^2}\|\Lam^s v^2\|_{L^2}\\
&\qquad
+\epsilon\|((A(D)-\epsilon\z)\p_1\Lam^s v^1\|_{H^{-1}}\|(1+c\mu\D)\bigl([g(D),\,v^2]\Lam^s v^2\bigr)\|_{H^1}\\
&\lesssim\epsilon\bigl(1+\epsilon\|\z\|_{H^s}+\epsilon\|v^2\|_{H^s}\bigr)\|v^1\|_{X^s_{\mu^2}}\|v^2\|_{X^s_{\mu^2}}^2,
\end{aligned}\eeno
where we also used the fact $s\geq t_0+2>3$. Similar estimates hold for $\bigl((a_{231}^*+a_{321})\Lam^s v^1\,|\,\Lam^s v^2\bigr)_2$
and $\bigl((a_{233}^*+a_{323})\Lam^s v^1\,|\,\Lam^s v^2\bigr)_2$. Using \eqref{interpolation}, we have
\beq\label{est 16}
|\bigl(a_{23}\Lam^sv^2\,|\,\Lam^sv^1\bigr)_2+\bigl(a_{32}\Lam^sv^1\,|\,\Lam^sv^2\bigr)_2|
\lesssim\epsilon\bigl(1+\epsilon\|\z\|_{X^s_{\mu^3}}+\epsilon\|\vv v\|_{X^s_{\mu^3}}+\epsilon^2\|\vv v\|_{X^s_{\mu^3}}^2\bigr)\|\vv v\|_{X^s_{\mu^3}}^3.
\eeq

\smallskip

Thanks to \eqref{est 9}, \eqref{est 12} and \eqref{est 15}, we could obtain the estimate for $\bigl(S_{\vv V}(D)M(\vv V,D)\Lam^s\vv V\,|\,\Lam^s\vv V\bigr)_2=\bigl(\mathcal{A}_{\vv V}(D)\Lam^s\vv V\,|\,\Lam^s\vv V\bigr)_2$. Since the same estimate holds for
$\bigl(S_{\vv V}(D)^*M(\vv V,D)\Lam^s\vv V\,|\,\Lam^s\vv V\bigr)_2$, using \eqref{equivalent 1}, we arrive at
\beq\label{estimate for I2}
|I_2|\lesssim\epsilon\bigl(1+\epsilon^2 E_s(\vv V)\bigr)\bigl(E_s(\vv V)\bigr)^{\f32}.
\eeq

\smallskip

{\it Step 1.3. Estimate on $I$.}
Due to \eqref{estimate for I1} and \eqref{estimate for I2}, we obtain
 \beq\label{estimate for I}
|I|\lesssim\epsilon\bigl(1+\epsilon^2 E_s(\vv V)\bigr)\bigl(E_s(\vv V)\bigr)^{\f32}.
\eeq

\medskip

{\bf Step 2. Estimate on $II$.} Thanks to the expression of $S_{\vv V}(D)$, we have
\beno\begin{aligned}
&|II|\lesssim\mu\epsilon\|[\D,\vv v]g(D)\Lam^s\vv V\|_{H^1}\|(1+c\mu\D)\Lam^s\p_t\vv V\|_{H^{-1}}
+\mu\epsilon\|[\D,\z]\Lam^s\vv V\|_{H^1}\|(1+c\mu\D)\Lam^s\p_t\vv V\|_{H^{-1}}\\
&\qquad
+\mu\epsilon^2\sum_{i,j=1,2}\|[\D,v^iv^j]\Lam^s\vv V\|_{H^1}\|(g(D)-1)\Lam^s\p_t\vv V\|_{H^{-1}},
\end{aligned}\eeno
which along with \eqref{g 1},  \eqref{interpolation} and \eqref{equivalent 1}, noticing that $s\geq t_0+2>3$, we have
 \beq\label{estimate for II}\begin{aligned}
&|II|\lesssim\epsilon\bigl(1+\epsilon\|\vv V\|_{H^s}\bigr)\|\vv V\|_{X^s_{\mu^2}}^2\|\p_t\vv V\|_{X^{s-1}_{\mu^2}}\\
&\lesssim\epsilon\bigl(1+\epsilon (E_s(\vv V))^{\f12}\bigr)E_s(\vv V)\|\p_t\vv V\|_{X^{s-1}_{\mu^2}}.
\end{aligned}\eeq

\medskip

{\bf Step 3. Estimate on $III$.} Thanks to the expression of $S_{\vv V}(D)$,  using \eqref{g 1} and \eqref{g 2}, we have
\beno
|III|\lesssim\|(1-b\mu\D)\Lam^s\vv V\|_{L^2}\|\p_tS_{\vv V}(D)\Lam^s\vv V\|_{L^2}
\lesssim\epsilon(1+\|\vv v\|_{H^{t_0}})\|\p_t\vv V\|_{H^{t_0}}\|\vv V\|_{X^s_{\mu^2}}^2.
\eeno
Due to \eqref{interpolation} and \eqref{equivalent 1}, noticing that $s\geq t_0+2>3$, we have
 \beq\label{estimate for III}
|III|\lesssim\epsilon\bigl(1+\epsilon (E_s(\vv V))^{\f12}\bigr)E_s(\vv V)\|\p_t\vv V\|_{H^{s-2}}.
\eeq

\medskip

{\bf Step 4. The {\it a priori} energy estimate.}  Thanks to \eqref{estimate for I}, \eqref{estimate for II} and \eqref{estimate for III}, we deduce from \eqref{est 3} that
\beq\label{est 17}
\f{d}{dt} E_s(\vv V)\lesssim\epsilon\bigl(1+\epsilon^2 E_s(\vv V)\bigr)\bigl(E_s(\vv V)\bigr)^{\f32}
+\epsilon\bigl(1+\epsilon (E_s(\vv V))^{\f12}\bigr)E_s(\vv V)\|\p_t\vv V\|_{X^{s-1}_{\mu^2}}.
\eeq

Going back to the equation \eqref{B-FD 1}, using \eqref{g 1}, \eqref{product}, \eqref{est 6}, \eqref{interpolation} and \eqref{equivalent 1}, we have
\beno\begin{aligned}
&\|\p_t\vv V\|_{X^{s-1}_{\mu^2}}\sim\|(1-b\mu\D)\p_t\vv V\|_{H^{s-1}}\lesssim\|M(\vv V, D)\vv V\|_{H^{s-1}}\\
&\lesssim\bigl(1+\epsilon\|\vv V\|_{H^{s}}\bigr)\|\vv V\|_{X^{s}_{\mu^2}}\lesssim\bigl(1+\epsilon (E_s(\vv V))^{\f12}\bigr)(E_s(\vv V))^{\f12},
\end{aligned}\eeno
which along with \eqref{est 17} implies
\beq\label{priori 1}
\f{d}{dt} E_s(\vv V)\lesssim\epsilon\bigl(1+\epsilon^2 E_s(\vv V)\bigr)\bigl(E_s(\vv V)\bigr)^{\f32}.
\eeq
This is exactly \eqref{priori energy estimate}.

\subsection{A priori estimates for the case: $ b>0,\, d=0, \, a\leq 0,\, c<0$.}
In this case, one could check that
\beq\label{equivalent 2}
E_s(\vv V)\sim\mathcal{E}_s{(t)}\eqdefa\|\z(t)\|_{X^s_{\mu^4}}^2+\|\vv v(t)\|_{X^s_{\mu^{3}}}^2
\eeq
for any $\epsilon\leq\wt\epsilon_1$ and $\mu\leq\wt\mu$ with $\wt\epsilon_1$ and $\wt\mu$ being sufficiently small. We postpone the proof of \eqref{equivalent 2} to the Appendix.

Since the proof of \eqref{priori energy estimate} of this case is similar to that of case $b\neq d,\, b>0,\, d>0, \, a\leq 0,\, c<0$, we only sketch it.

A direct energy estimate shows that
\beq\label{est 18}\begin{aligned}
&\f{d}{dt}E_s(\vv V)=\bigl((1-b\mu\Delta)\Lam^s\p_t\vv V\,|\,(S_{\vv V}(D)+S_{\vv V}(D)^*)\Lam^s\vv V\bigr)_2\\
&\quad-b\mu([S_{\vv V}(D)^*, \Delta]\Lam^s\vv V\,|\,\Lam^s\p_t\vv V)_2
+\bigl((1-b\mu\Delta)\Lam^s\vv V\,|\,\p_tS_{\vv V}(D)\Lam^s\vv V\bigr)_2\\
&\eqdefa I+II+III,
\end{aligned}\eeq
where $S_{\vv V}(D)^*$ is the adjoint operator of $S_{\vv V}(D)$.

{\bf Step 1. Estimate on $I$. } Using \eqref{B-FD 1}, we have
\beq\label{est 19}\begin{aligned}
&I=-\bigl([\Lam^s,M(\vv V,D)]\vv V\,|\,(S_{\vv V}(D)+S_{\vv V}(D)^*)\Lam^s\vv V\bigr)_2\\
&\quad
-\bigl((S_{\vv V}(D)+S_{\vv V}(D)^*)\bigl(M(\vv V,D)\Lam^s\vv V\Bigr)\,|\,\Lam^s\vv V\bigr)_2\eqdefa I_1+I_2.
\end{aligned}\eeq

{\it Step 1.1. Estimate on $I_1$.} Firstly, notice that \eqref{est a} also holds for the present case with $g(D)=1-b\mu\D$.
Similar derivation as \eqref{estimate for I1}, we have
\beq\label{estimate for I1 d=0}\begin{aligned}
|I_1|&\lesssim\epsilon\bigl(1+\epsilon\|\z\|_{X^s_{\mu^4}}+\epsilon\|\vv v\|_{X^s_{\mu^3}}\bigr)^2\|\vv v\|_{X^s_{\mu^3}}
\bigl(\|\z\|_{X^s_{\mu^4}}^2+\|\vv v\|_{X^s_{\mu^3}}^2\bigr)\lesssim\epsilon\bigl(1+\epsilon^2 E_s(\vv V)\bigr)\bigl(E_s(\vv V)\bigr)^{\f32}.
\end{aligned}\eeq

\smallskip

{\it Step 1.2. Estimate on $I_2$.} In order to estimate $I_2$, we first calculate $S_{\vv V}(D)M(\vv V,D)\eqdefa\mathcal{A}_{\vv V}(D)\eqdefa(a_{ij})_{i,j=1,2,3}$. We point out that $a_{ij}$ has the same expression as that in \eqref{expression of A} with $g(D)=1-b\mu\D$.

Now, we estimate $\bigl(S_{\vv V}(D)M(\vv V,D)\Lam^s\vv V\,|\,\Lam^s\vv V\bigr)_2=\bigl(\mathcal{A}_{\vv V}(D)\Lam^s\vv V\,|\,\Lam^s\vv V\bigr)_2$ term by term.

Following similar derivation as that of case $b\neq d,\,b>0,\,d>0,\,a\leq0,\,c<0$,  integrating by parts, we first have
\beq\label{est 20}\begin{aligned}
&|\bigl(a_{11}\Lam^s\z\,|\,\Lam^s\z\bigr)_2|\lesssim\epsilon\|\vv v\|_{X^s_{\mu^3}}\|\z\|_{X^s_{\mu^4}}^2,\\
&|\bigl(a_{22}\Lam^sv^1\,|\,\Lam^s v^1\bigr)_2|+|\bigl(a_{33}\Lam^sv^2\,|\,\Lam^s v^2\bigr)_2|
\lesssim\epsilon\bigl(1+\epsilon\|\z\|_{X^s_{\mu^4}}+\epsilon\|\vv v\|_{X^s_{\mu^3}}\bigr)^2\|\vv v\|_{X^s_{\mu^3}}^3.
\end{aligned}\eeq

{\it For $a_{12}$ and $a_{21}$}, firstly, noticing that \eqref{est 10a}  holds for $a_{121}^*+a_{212}$ with $g(D)=1-b\mu\D$, we have similar estimate as \eqref{est 10} as follows
\beq\label{est 21}
|\bigl((a_{121}^*+a_{211})\Lam^s\z\,|\,\Lam^s v^1\bigr)_2|\lesssim\epsilon\|\z\|_{X^{t_0+1}_{\mu^2}}\|\z\|_{X^s_{\mu^4}}\|\vv v\|_{X^s_{\mu^2}}\lesssim\epsilon\|\z\|_{X^s_{\mu^4}}^2\|\vv v\|_{X^s_{\mu^3}}.
\eeq

For $a_{122}^*+a_{212}+a_{213}$, there hold \eqref{est 22a} and \eqref{est 22b} with  $g(D)=1-b\mu\D$. Similarly as \eqref{est 11}, we have
\beno
|\bigl((a_{122}^*+a_{212}+a_{213})\Lam^s\z\,|\,\Lam^s v^1\bigr)_2|\lesssim\epsilon^2\|v^1\|_{H^{t_0+1}}\|\vv v\|_{X^{t_0+2}_{\mu^2}}\|\z\|_{X^s_{\mu^4}}\|v^1\|_{X^s_{\mu^2}}\lesssim\epsilon^2\|\z\|_{X^s_{\mu^4}}\|\vv v\|_{X^s_{\mu^3}}^3,
\eeno
which along with \eqref{est 21} implies
\beq\label{est 22}
|\bigl(a_{12}\Lam^sv^1\,|\,\Lam^s\z\bigr)_2+\bigl(a_{21}\Lam^s\z\,|\,\Lam^s v^1\bigr)_2|\lesssim\epsilon\bigl(1+\epsilon\|\vv v\|_{X^s_{\mu^3}}\bigr)\|\vv v\|_{X^s_{\mu^3}}\bigl(\|\z\|_{X^s_{\mu^4}}^2+\|\vv v\|_{X^s_{\mu^3}}^2\bigr).
\eeq
The same estimate holds for $\bigl(a_{13}\Lam^sv^2\,|\,\Lam^s\z\bigr)_2+\bigl(a_{31}\Lam^s\z\,|\,\Lam^s v^2\bigr)_2$.

{\it For $a_{23}$ and $a_{32}$}, there also holds \eqref{est 23a} with $g(D)=1-b\mu\D$. Then we get
 \beq\label{est 23}
|\bigl(a_{23}\Lam^sv^2\,|\,\Lam^sv^1\bigr)_2+\bigl(a_{32}\Lam^sv^1\,|\,\Lam^s v^2\bigr)_2|\lesssim\epsilon\bigl(1+\epsilon\|\z\|_{X^s_{\mu^4}}+\epsilon\|\vv v\|_{X^s_{\mu^3}}\bigr)^2\|\vv v\|_{X^s_{\mu^3}}\bigl(\|\z\|_{X^s_{\mu^4}}^2+\|\vv v\|_{X^s_{\mu^3}}^2\bigr).
\eeq

\smallskip

Thanks to \eqref{est 21}, \eqref{est 22} and \eqref{est 23}, we obtain the bound of $\bigl(S_{\vv V}(D)M(\vv V,D)\Lam^s\vv V\,|\,\Lam^s\vv V\bigr)_2$. The same estimate holds for $\bigl(S_{\vv V}(D)^*M(\vv V,D)\Lam^s\vv V\,|\,\Lam^s\vv V\bigr)_2$. Using \eqref{equivalent 2}, we arrive at
\beq\label{estimate for I2 d=0}\begin{aligned}
|I_2|\lesssim\epsilon\bigl(1+\epsilon^2 E_s(\vv V)\bigr)\bigl(E_s(\vv V)\bigr)^{\f32}.
\end{aligned}\eeq

Due to \eqref{estimate for I1 d=0} and \eqref{estimate for I2 d=0}, we obtain
\beq\label{estimate for I d=0}\begin{aligned}
|I|\lesssim\epsilon\bigl(1+\epsilon^2 E_s(\vv V)\bigr)\bigl(E_s(\vv V)\bigr)^{\f32}.
\end{aligned}\eeq

\medskip

{\bf Step 2. Estimate on $II$.} Thanks to the expression of $S_{\vv V}(D)$, we have
\beno\begin{aligned}
&II=b\epsilon(1-\gamma)\gamma\mu\bigl([\D,\,\vv v]\cdot(1-b\mu\D)\Lam^s\vv v\,|\,(1+c\mu\D)\Lam^s\p_t\z\bigr)_2
+b\epsilon(1-\gamma)\gamma\mu\bigl([\D,\,\vv v](1-b\mu\D)\Lam^s\z\,|\,(1+c\mu\D)\Lam^s\p_t\vv v\bigr)_2\\
&\qquad+b\epsilon(1-\gamma)\gamma\mu\bigl([\D,\,\z]\Lam^s\vv v\,|\,(1+c\mu\D)\Lam^s\p_t\vv v\bigr)_2
+b^2\epsilon^2\mu^2\sum_{i,j=1,2}\bigl([\D,\,v^iv^j]\Lam^sv^j\,|\,\D\Lam^s\p_tv^i\bigr)_2,
\end{aligned}\eeno
which along with \eqref{interpolation} implies
\beno\begin{aligned}
&|II|\lesssim\epsilon\|\vv v\|_{X^{t_0+1}_{\mu}}\|\vv v\|_{X^s_{\mu^3}}\|\p_t\z\|_{X^{s-1}_{\mu^3}}
+\epsilon\|\vv v\|_{X^{t_0+1}_{\mu}}\|\z\|_{X^s_{\mu^4}}\|\p_t\vv v\|_{X^{s-1}_{\mu^2}}\\
&\qquad+\epsilon\|\z\|_{X^{t_0+1}_{\mu}}\|\vv v\|_{X^s_{\mu^2}}\|\p_t\vv v\|_{X^{s-1}_{\mu^2}}
+\epsilon^2\|\vv v\|_{H^{t_0+1}}\|\vv v\|_{X^{t_0+1}_{\mu^2}}\|\vv v\|_{X^s_{\mu^2}}\|\p_t\vv v\|_{X^{s-1}_{\mu^2}}.
\end{aligned}\eeno
Since $s\geq t_0+2>3$, using \eqref{interpolation} and \eqref{equivalent 2}, we obtain
\beq\label{estimate for II d=0}\begin{aligned}
|II|\lesssim\epsilon\bigl(1+\epsilon\bigl(E_s(\vv V)\bigr)^{\f12}\bigr)E_s(\vv V)\bigl(\|\p_t\z\|_{X^{s-1}_{\mu^3}}+\|\p_t\vv v\|_{X^{s-1}_{\mu^2}}\bigr).
\end{aligned}\eeq

\medskip

{\bf Step 3. Estimate on $III$.} Thanks to the expression of $S_{\vv V}(D)$, we have
\beno\begin{aligned}
&III=-\epsilon(1-\gamma)\bigl((1-b\mu\D)^2\Lam^s\z\,|\,\p_t\vv v\cdot(1+c\mu\D)\Lam^s\vv v\bigr)_2
-\epsilon(1-\gamma)\bigl((1-b\mu\D)\Lam^s\vv v\,|\,(1-b\mu\D)\bigl(\p_t\vv v(1+c\mu\D)\Lam^s\z\bigr)\bigr)_2\\
&\qquad-\epsilon(1-\gamma)\gamma\bigl((1-b\mu\D)\Lam^s\vv v\,|\,\p_t\z(1+c\mu\D)\Lam^s\vv v\bigr)_2
-b\mu\epsilon\sum_{i,j=1,2}\bigl((1-b\mu\D)\Lam^sv^i\,|\,\p_t(v^iv^j)(1+c\mu\D)\Lam^sv^j\bigr)_2,
\end{aligned}\eeno
which along with \eqref{interpolation} implies
\beno\begin{aligned}
&|III|\lesssim\epsilon\|\z\|_{X^s_{\mu^4}}\|\vv v\|_{X^s_{\mu^3}}\|\p_t\z\|_{H^{t_0+1}}+\epsilon\bigl(1+\epsilon\|\vv v\|_{H^{t_0}}\bigr)\|\vv v\|_{X^s_{\mu^3}}^2\bigl(\|\p_t\z\|_{H^{t_0}}+\|\p_t\vv v\|_{H^{t_0}}\bigr).
\end{aligned}\eeno
Since $s\geq t_0+2>3$, using \eqref{interpolation} and \eqref{equivalent 2}, we obtain
\beq\label{estimate for III d=0}\begin{aligned}
|III|\lesssim\epsilon\bigl(1+\epsilon\bigl(E_s(\vv V)\bigr)^{\f12}\bigr)E_s(\vv V)\bigl(\|\p_t\z\|_{H^{s-1}}+\|\p_t\vv v\|_{H^{s-1}}\bigr).
\end{aligned}\eeq

\medskip

{\bf Step 4. The {\it a priori} energy estimate.}  Thanks to \eqref{estimate for I d=0}, \eqref{estimate for II d=0} and \eqref{estimate for III d=0}, we deduce from \eqref{est 18} that
\beq\label{est 24}
\f{d}{dt} E_s(\vv V)\lesssim\epsilon\bigl(1+\epsilon^2 E_s(\vv V)\bigr)\bigl(E_s(\vv V)\bigr)^{\f32}
+\epsilon\bigl(1+\epsilon (E_s(\vv V))^{\f12}\bigr)E_s(\vv V)\bigl(\|\p_t\z\|_{X^{s-1}_{\mu^3}}+\|\p_t\vv v\|_{X^{s-1}_{\mu^2}}\bigr).
\eeq

Going back to the equation \eqref{B-FD}, using \eqref{g 1}, \eqref{product}, \eqref{est 6} and \eqref{interpolation}, we have
\beno\begin{aligned}
&\|\p_t\z\|_{X^{s-1}_{\mu^3}}\lesssim\|\bigl(A(D)-\epsilon\z\bigr)\vv v\|_{X^{s-1}_{\mu}}
\lesssim\bigl(1+\epsilon\|\z\|_{X^s_\mu}\bigr)\|\vv v\|_{X^s_{\mu^3}},\\
&\|\p_t\vv v\|_{X^{s-1}_{\mu^2}}\lesssim\|(1+c\mu\D)\z\|_{X^{s}_{\mu^2}}+\epsilon\|\vv v\cdot\vv v\|_{X^{s}_{\mu^2}}
\lesssim\|\z\|_{X^s_{\mu^4}}+\epsilon\|\vv v\|_{X^s_{\mu^3}}^2,
\end{aligned}\eeno
which along with  \eqref{interpolation} and \eqref{equivalent 2} implies
\beq\label{est 25}
\|\p_t\z\|_{X^{s-1}_{\mu^3}}+\|\p_t\vv v\|_{X^{s-1}_{\mu^2}}\lesssim\bigl(1+\epsilon (E_s(\vv V))^{\f12}\bigr)(E_s(\vv V))^{\f12}.
\eeq

Due to \eqref{est 24} and \eqref{est 25}, we get
\beq\label{priori 2}
\f{d}{dt} E_s(\vv V)\lesssim\epsilon\bigl(1+\epsilon^2 E_s(\vv V)\bigr)\bigl(E_s(\vv V)\bigr)^{\f32}.
\eeq
This is exactly \eqref{priori energy estimate}.

\subsection{A priori estimates for the case: $ b=d=0, \, a\leq 0,\, c<0$.}
In this case, the equation \eqref{B-FD} is equivalent to the following condensed system
\beq\label{B-FD 3}
\p_t\vv V+M(\vv V,D)\vv V=\vv 0,
\eeq
where $M(\vv V,D)$ is defined in \eqref{M 1} with $g(D)=1$. The symmetrizer $S_{\vv V}(D)$ of $M(\vv V,D)$ is defined by \eqref{S for b=d>0}.

Defining the associated energy functional as
\beq\label{E 3}
E_s(\vv V)\eqdefa\bigl(\Lam^s\vv V\,|\,S_{\vv V}(D)\Lam^s\vv V\bigr)_2,
\eeq
one could check that
\beq\label{equivalent 3}
E_s(\vv V)\sim\mathcal{E}_s{(t)}\eqdefa\|\z(t)\|_{X^s_{\mu}}^2+\|\vv v(t)\|_{X^s_{\mu}}^2,
\eeq
for any $\epsilon\leq\wt\epsilon_1$ and $\mu\leq\wt\mu$ with $\wt\epsilon_1$ and $\wt\mu$ being sufficiently small. We postpone the proof of \eqref{equivalent 3} to the Appendix.

A direct energy estimate shows that
\beq\label{est 26}\begin{aligned}
&\f{d}{dt}E_s(\vv V)=\bigl(\Lam^s\p_t\vv V\,|\,(S_{\vv V}(D)+S_{\vv V}(D)^*)\Lam^s\vv V\bigr)_2
+\bigl(\Lam^s\vv V\,|\,\p_tS_{\vv V}(D)\Lam^s\vv V\bigr)_2\eqdefa I+II.
\end{aligned}\eeq

\medskip

{\bf Step 1. Estimate on $I$.} Thanks to \eqref{B-FD 3}, we have
\beq\label{est 27}\begin{aligned}
&I=-\bigl([\Lam^s,M(\vv V,D)]\vv V\,|\,(S_{\vv V}(D)+S_{\vv V}(D)^*)\Lam^s\vv V\bigr)_2\\
&\quad
-\bigl((S_{\vv V}(D)+S_{\vv V}(D)^*)\bigl(M(\vv V,D)\Lam^s\vv V\bigr)\,|\,\Lam^s\vv V\bigr)_2\eqdefa I_1+I_2.
\end{aligned}\eeq

\smallskip

{\it Step 1.1. Estimate on $I_1$.} By the expressions of $M(\vv V,D)$ and $S_{\vv V}(D)$ in \eqref{M 1} and \eqref{S for b=d>0} with $g(D)=1$, we first have
\beno\begin{aligned}
&\quad\bigl([\Lam^s,M(\vv V,D)]\vv V\,|\,S_{\vv V}(D)\Lam^s\vv V\bigr)_2\\
&=-\epsilon(1-\gamma)\bigl([\Lam^s,\vv v]\cdot\na\z\,|\,(1+c\mu\D)\Lam^s\z\bigr)_2-\epsilon(1-\gamma)\bigl([\Lam^s,\z]\na\cdot\vv v\,|\,(1+c\mu\D)\Lam^s\z\bigr)_2\\
&\qquad+\f{\epsilon^2}{\gamma}\bigl([\Lam^s,\vv v]\cdot\na\z+[\Lam^s,\z]\na\cdot\vv v\,|\,\vv v\cdot\Lam^s\vv v\bigr)_2+\f{\epsilon^2}{\gamma}\sum_{j=1,2}\bigl([\Lam^s,\vv v]\cdot\p_j\vv v\,|\,v^j\Lam^s\z\bigr)_2\\
&\qquad
-\f{\epsilon}{\gamma}\sum_{j=1,2}\bigl([\Lam^s,\vv v]\cdot\p_j\vv v\,|\,(A(D)-\epsilon\z)\Lam^sv^j\bigr)_2
\eqdefa I_{11}+I_{12}+I_{13}+I_{14}+I_{15}.
\end{aligned}\eeno

For $I_{11}$, integration by parts gives rise to
\beno\begin{aligned}
&|I_{11}|\lesssim\epsilon\|[\Lam^s,\vv v]\cdot\na\z\|_{L^2}\|\Lam^s\z\|_{L^2}+\epsilon\mu\|\na([\Lam^s,\vv v]\cdot\na\z)\|_{L^2}\|\na\Lam^s\z\|_{L^2},
\end{aligned}\eeno
which along with \eqref{commutator 1} implies
\beq\label{est 28}
|I_{11}|\lesssim\epsilon\|\vv v\|_{X^s_\mu}\|\z\|_{X^s_\mu}^2.
\eeq
Similar estimates hold for $I_{12}$, $I_{13}$, $I_{14}$ and $I_{15}$. Then using \eqref{equivalent 3}, we obtain
\beq\label{estimate for I1 b=d=0}
|I_1|\lesssim\epsilon\bigl(1+\epsilon\|\vv v\|_{X^s_\mu}\bigr)\|\vv v\|_{X^s_\mu}\bigl(\|\z\|_{X^s_\mu}^2+\|\vv v\|_{X^s_\mu}^2\bigr)
\lesssim\epsilon\bigl(1+\epsilon (E_s(\vv V))^{\f12}\bigr)(E_s(\vv V))^{\f32}.
\eeq

\smallskip

{\it Step 1.2. Estimate on $I_2$.} By the expressions of $M(\vv V,D)$ and $S_{\vv V}(D)$ in \eqref{M 1} and \eqref{S for b=d>0} with $g(D)=1$, we calculate $S_{\vv V}(D)M(\vv V,D)=(a_{ij})_{i,j=1,2}$ as follows:
\beno\begin{aligned}
&a_{11}=-\epsilon(1-\gamma)[(1+c\mu\D)(\vv v\cdot\na)+\vv v\cdot\na(1+c\mu\D)],\\
&a_{12}=(1-\gamma)(1+c\mu\D)\bigl((A(D)-\epsilon\z)\p_1\bigr)+\f{\epsilon^2}{\gamma}v^1\vv v\cdot\na,\\
&a_{13}=(1-\gamma)(1+c\mu\D)\bigl((A(D)-\epsilon\z)\p_2\bigr)+\f{\epsilon^2}{\gamma}v^2\vv v\cdot\na,\\
&a_{21}=(1-\gamma)(A(D)-\epsilon\z)(1+c\mu\D)\p_1+\f{\epsilon^2}{\gamma}v^1\vv v\cdot\na,\\
&a_{22}=-\f{\epsilon}{\gamma}\bigl[v^1(A(D)-\epsilon\z)\p_1+(A(D)-\epsilon\z)(v^1\p_1)\bigr],\\
&a_{23}=-\f{\epsilon}{\gamma}\bigl[v^1(A(D)-\epsilon\z)\p_2+(A(D)-\epsilon\z)(v^2\p_1)\bigr],\\
&a_{31}=(1-\gamma)(A(D)-\epsilon\z)(1+c\mu\D)\p_2+\f{\epsilon^2}{\gamma}v^2\vv v\cdot\na\\
&a_{32}=-\f{\epsilon}{\gamma}\bigl[v^2(A(D)-\epsilon\z)\p_1+(A(D)-\epsilon\z)(v^1\p_2)\bigr],\\
&a_{33}=-\f{\epsilon}{\gamma}\bigl[v^2(A(D)-\epsilon\z)\p_2+(A(D)-\epsilon\z)(v^2\p_2)\bigr].
\end{aligned}\eeno

Now, we calculate $\bigl(S_{\vv V}(D)\bigl(M(\vv V,D)\Lam^s\vv V\bigr)\,|\,\Lam^s\vv V\bigr)_2$.

{\it For $a_{11}$,} integration by parts gives rise to
\beno\begin{aligned}
&\bigl(a_{11}\Lam^s\z\,|\,\Lam^s\z\bigr)_2=\epsilon(1-\gamma)\bigl(\na\cdot\vv v\Lam^s\z\,|\,\Lam^s\z\bigr)_2
-c\mu\epsilon(1-\gamma)\bigl(\na\cdot\vv v\na\Lam^s\z\,|\,\na\Lam^s\z\bigr)_2\\
&\quad
+c\mu\epsilon(1-\gamma)\sum_{j=1,2}\{\bigl(\p_j\vv v\cdot\na\Lam^s\z\,|\,\p_j\Lam^s\z\bigr)_2-\bigl(\na\cdot(\p_j\vv v\Lam^s\z)\,|\,\p_j\Lam^s\z\bigr)_2\},
\end{aligned}\eeno
which implies
\beq\label{est 29}
|\bigl(a_{11}\Lam^s\z\,|\,\Lam^s\z\bigr)_2|\lesssim\epsilon\|\vv v\|_{X^s_\mu}\|\z\|_{X^s_\mu}^2.
\eeq

{\it For $a_{22}$,} we first deal with one term involving $A(D)$ as follows
\beno\begin{aligned}
&\quad-\f{\epsilon}{\gamma^3}\f{\mu}{\mu_2}\bigl(v^1\sigma(D)^2\p_1\Lam^sv^1\,|\,\Lam^sv^1\bigr)_2\\
&
=\f{\epsilon}{\gamma^3}\f{\mu}{\mu_2}\{\bigl(\sigma(D)\Lam^sv^1\,|\,\p_1\bigl([\sigma(D),v^1]\Lam^sv^1\bigr)\bigr)_2
+\f12\bigl(\p_1v^1\sigma(D)\Lam^sv^1\,|\,\sigma(D)\Lam^sv^1\bigr)\bigr)_2\},
\end{aligned}\eeno
which along with \eqref{sigma 1} and \eqref{commutator of sigma} implies
\beno
\f{\epsilon}{\gamma^3}\f{\mu}{\mu_2}|\bigl(v^1\sigma(D)^2\p_1\Lam^sv^1\,|\,\Lam^sv^1\bigr)_2|\lesssim\epsilon\|\vv v\|_{X^s_\mu}^3.
\eeno
Similar estimate holds for the other terms in $\bigl(a_{22}\Lam^sv^1\,|\,\Lam^sv^1\bigr)_2$. Then we obtain
\beq\label{est 30}
|\bigl(a_{22}\Lam^sv^1\,|\,\Lam^sv^1\bigr)_2|\lesssim\epsilon(1+\epsilon\|\z\|_{X^s_\mu})\|\vv v\|_{X^s_\mu}^3.
\eeq
The same estimate holds for $\bigl(a_{33}\Lam^sv^2\,|\,\Lam^sv^2\bigr)_2$.

{\it For $a_{12}$ and $a_{21}$,}  it is easy to check that
\beno\begin{aligned}
&a_{12}^*+a_{21}=\epsilon(1-\gamma)\p_1\z(1+c\mu\D)-\f{\epsilon^2}{\gamma}\na\cdot(v^1\vv v),
\end{aligned}\eeno
which implies
\beq\label{est 31}\begin{aligned}
&|\bigl(a_{12}\Lam^sv^1\,|\,\Lam^s\z\bigr)_2+\bigl(a_{21}\Lam^s\z\,|\,\Lam^sv^1\bigr)_2|
\lesssim\epsilon\|\z\|_{X^{t_0+1}_\mu}\|\vv v\|_{X^s_\mu}\|\z\|_{X^s_\mu}+\epsilon^2\|\vv v\|_{H^{t_0+1}}^2\|\vv v\|_{H^s}\|\z\|_{H^s}\\
&\lesssim\epsilon(1+\epsilon\|\vv v\|_{X^s_\mu})\|\z\|_{X^s_\mu}\bigl(\|\z\|_{X^s_\mu}^2+\|\vv v\|_{X^s_\mu}^2\bigr).
\end{aligned}\eeq
The same estimate holds for $\bigl(a_{13}\Lam^sv^2\,|\,\Lam^s\z\bigr)_2+\bigl(a_{31}\Lam^s\z\,|\,\Lam^sv^2\bigr)_2$.

{\it For $a_{23}$ and $a_{32}$,}  it is easy to check that
\beno\begin{aligned}
&a_{23}^*+a_{32}=\f{\epsilon}{\gamma}\p_1v^2A(D)-\f{\epsilon^2}{\gamma}\p_1(v^2\z)\cdot
+\f{\epsilon}{\gamma}A(D)(\p_2v^1\cdot)-\f{\epsilon^2}{\gamma}\p_2(v^1\z)\cdot.
\end{aligned}\eeno
Thanks to the expression of $A(D)$ in \eqref{notation}, using \eqref{sigma 1} and \eqref{interpolation}, we get
\beq\label{est 32}\begin{aligned}
&|\bigl(a_{23}\Lam^sv^2\,|\,\Lam^sv^1\bigr)_2+\bigl(a_{32}\Lam^sv^1\,|\,\Lam^sv^2\bigr)_2|
\lesssim\epsilon(1+\epsilon\|\z\|_{X^s_\mu})\|\vv v\|_{X^s_\mu}^3.
\end{aligned}\eeq

\smallskip

Combining \eqref{est 29}, \eqref{est 30}, \eqref{est 31} and \eqref{est 32}, we obtain the estimate for $\bigl(S_{\vv V}(D)\bigl(M(\vv V,D)\Lam^s\vv V\bigr)\,|\,\Lam^s\vv V\bigr)_2$. The same estimate holds for $\bigl(S_{\vv V}(D)^*\bigl(M(\vv V,D)\Lam^s\vv V\bigr)\,|\,\Lam^s\vv V\bigr)_2$. Then using \eqref{equivalent 3}, we obtain
\beq\label{estimate for I2 b=d=0}
|I_2|\lesssim\epsilon\bigl(1+\epsilon (E_s(\vv V))^{\f12}\bigr)(E_s(\vv V))^{\f32},
\eeq
which along with \eqref{estimate for I1 b=d=0} implies
\beq\label{estimate for I b=d=0}
|I|\lesssim\epsilon\bigl(1+\epsilon (E_s(\vv V))^{\f12}\bigr)(E_s(\vv V))^{\f32}.
\eeq

\medskip

{\bf Step 2. Estimate on $II$.} Thanks to the expression of $S_{\vv V}(D)$ in \eqref{S for b=d>0}, we have
\beno
II=-\epsilon\bigl(\Lam^s\z\,|\,\p_t\vv v\cdot\Lam^s\vv v\bigr)_2-\epsilon\bigl(\Lam^s\vv v\,|\,\p_t\vv v\Lam^s\z\bigr)_2,
\eeno
which along with \eqref{equivalent 3} implies
\beq\label{estimate for II b=d=0}
|II|\lesssim\epsilon\bigl(\|\z\|_{H^s}^2+\|\vv v\|_{H^s}^2\bigr)\bigl(\|\p_t\z\|_{H^{s-2}}+\|\p_t\vv v\|_{H^{s-2}}\bigr)
\lesssim\epsilon E_s(\vv V)\|\p_t\vv V\|_{H^{s-2}},
\eeq
where we used the fact that $s\geq t_0+2$.

\medskip

{\bf Step 3. The {\it a priori} energy estimate.} Thanks to \eqref{B-FD 3}, we have
\beno
\|\p_t\vv V\|_{H^{s-2}}\lesssim\bigl(1+\epsilon\|\z\|_{H^s}+\epsilon\|\vv v\|_{H^s}\bigr)\|\vv V\|_{X^s_\mu},
\eeno
which along with \eqref{estimate for I b=d=0},  \eqref{estimate for II b=d=0}, \eqref{est 26}  and \eqref{equivalent 3} implies
\beq\label{priori 3}
\f{d}{dt} E_s(\vv V)\lesssim\epsilon\bigl(1+\epsilon(E_s(\vv V))^{\f12}\bigr)\bigl(E_s(\vv V)\bigr)^{\f32}\lesssim\epsilon\bigl(1+\epsilon^2 E_s(\vv V)\bigr)\bigl(E_s(\vv V)\bigr)^{\f32}.
\eeq
This is exactly \eqref{priori energy estimate}.

\medskip

\begin{remark}
The a priori estimate \eqref{priori energy estimate} for the remain cases in Definition \ref{def2} can be treated  in a similar way as the cases in this section.
\end{remark}

\setcounter{equation}{0}
\section{Global existence for the Hamiltionian case $b=d>0,\,a\leq0,\,c<0$}
In this section, we shall prove Theorem \ref{Global existence for B-FD in case b=d} that is the global existence of solutions of \eqref{eqB-FD} with
$b=d>0,\,a\leq0,\,c<0$. We only discuss the two-dimensional case. The one-dimensional case follows in a similar way and actually it is considered in \cite{AS}.

\subsection{Hamiltonian structure for the Boussinesq-Full dispersion system when $b=d$.}
Recalling \eqref{B-FD}, we search a function $\mathcal{H}=\mathcal{H}(\z,\vv v)$ satisfying
\beq\label{H1}\begin{aligned}
&\f{\d\mathcal{H}}{\d\z}=(1-\gamma)(1+c\mu\D)\z-\f{\epsilon}{2\gamma}|\vv v|^2,\\
&\f{\d\mathcal{H}}{\d{\vv v}}=\f{1}{\gamma}(1-\epsilon\z)\vv v+\f{a\mu}{\gamma}\D\vv v+\f{1}{\gamma^2}\sqrt{\f{\mu}{\mu_2}}\sigma(D)\vv v
+\f{1}{\gamma^3}\f{\mu}{\mu_2}\sigma(D)^2\vv v.
\end{aligned}\eeq
Then we have
\beq\label{Hamiltonian}\begin{aligned}
\mathcal{H}(\z,\vv v)\eqdefa&\f12\int_{\R^2}\Bigl((1-\gamma)|\z|^2+\f{1}{\gamma}(1-\epsilon\z)|\vv v|^2-(1-\gamma)c\mu|\na\z|^2
-\f{a\mu}{\gamma}|\na\vv v|^2\\
&\qquad+\f{1}{\gamma^2}\sqrt{\f{\mu}{\mu_2}}|\sigma(D)^{\f12}\vv v|^2
+\f{1}{\gamma^3}\f{\mu}{\mu_2}|\sigma(D)\vv v|^2\Bigr)dx.
\end{aligned}\eeq

\begin{remark}
By the expression of $\mathcal{H}(\z,\vv v)$, and assuming that
\beq\label{H0}
1-\epsilon\z\geq H>0,
\eeq
we have for $a\leq 0,\,c<0$
\beno
\mathcal{H}(\z,\vv v)\sim\|\z\|_{X^0_\mu}^2+\|\vv v\|_{X^0_\mu}^2.
\eeno
However, condition \eqref{H0} could not be conserved in $X^0_\mu$ since $H^1(\R^2)$ is not embedding in $L^\infty(\R^2)$ contrary the one-dimensional case. Thus, the Hamiltonian is not obviously positive.
\end{remark}

Thanks to \eqref{B-FD} and \eqref{H1}, we have
\beq\label{H2}\begin{aligned}
&(1-b\mu\D)\p_t\z=-\na\cdot\f{\d\mathcal{H}}{\d{\vv v}},\\
&(1-d\mu\Delta)\p_t\vv v=-\na\f{\d\mathcal{H}}{\d{\z}}.
\end{aligned}\eeq
Due to \eqref{H2}, when $b=d$, \eqref{B-FD} is a Hamiltonian system that is given by
\beq\label{H B-FD}
\p_t
\left(
\begin{array}{c}
\z\\
\vv v
\end{array}\right)+J\na_{\z,\vv v}\mathcal{H}(\z,\vv v)=0.
\eeq
where
\beno
J=(1-b\mu\D)^{-1}\left(\begin{array}{cc}
0&\na\cdot\\
\na&\vv 0
\end{array}\right).
\eeno

Since $\mathcal{H}(\z,\vv v)$ is a Hamiltonian of \eqref{B-FD}, we have the following conservation law for \eqref{B-FD}.
\begin{lemma}\label{conservation}
When $b=d$, the smooth solution $(\z,\vv v)$ to \eqref{B-FD} satisfies
\beq\label{conservation law}
\f{d}{dt} \mathcal{H}(\z,\vv v)=0,
\eeq
where $\mathcal{H}(\z,\vv v)$ is a Hamiltonian defined by \eqref{Hamiltonian}.
\end{lemma}
\begin{proof}
Thanks to \eqref{Hamiltonian} and \eqref{H1}, we have
\beno\begin{aligned}
&\f{d}{dt}\mathcal{H}(\z,\vv v)=(\f{\d\mathcal{H}}{\d\z}\,|\,\p_t\z)_2+(\f{\d\mathcal{H}}{\d{\vv v}}\,|\,\p_t\vv v)_2
\end{aligned}\eeno
which along with \eqref{H2} implies
\beno\begin{aligned}
&\f{d}{dt}\mathcal{H}(\z,\vv v)=-\bigl(\f{\d\mathcal{H}}{\d\z}\,|\,(1-b\mu\D)^{-1}\na\cdot\f{\d\mathcal{H}}{\d{\vv v}}\bigr)_2-\bigl(\f{\d \mathcal{H}}{\d{\vv v}}\,|\,(1-d\mu\D)^{-1}\na\f{\d\mathcal{H}}{\d{\z}}\bigr)_2.
\end{aligned}\eeno
Since $b=d$, integration by parts gives rise to \eqref{conservation law}. The lemma is proved.
\end{proof}

\subsection{Local existence of the solutions to \eqref{B-FD} with $b=d>0,\,a\leq0,\,c<0$.}
In this subsection, we state the local existence and blow-up criteria for  \eqref{B-FD}-\eqref{initial data}.
\begin{proposition}\label{a priori for H B-FD}
Let  $b=d>0,\,a\leq0,\,c<0$, $\mu\sim\epsilon$. Assume that $(\z_0,\vv v_0)\in X^0_\mu(\R^2)\times X^0_\mu(\R^2)$. Then \eqref{B-FD}-\eqref{initial data} has a unique solution $(\z,\vv v)$ on $[0,T]$ for some $T>0$ so that  $(\z,\vv v)\in C(0,T; X^0_\mu(\R^2)\times X^0_\mu(\R^2))$  and
\beq\label{priori for H B-FD}
\max_{[0,T]}\bigl(\|\z\|_{X^0_\mu}+\|\vv v\|_{X^0_\mu}\bigr)\leq 2C_1\bigl(\|\z_0\|_{X^0_\mu}+\|\vv v_0\|_{X^0_\mu}\bigr),
\eeq
where $C_1>1$ is a constant.

Moreover, if $T^*$ is the lifespan to this solution and $T^*<\infty$, then
\beq\label{blow-up criteria}
\liminf_{t\rightarrow T^*}\bigl(\|\z(t)\|_{X^0_\mu}+\|\vv v(t)\|_{X^0_\mu}\bigr)=\infty.
\eeq

\end{proposition}
\begin{proof}
 We divide the proof into several steps.

{\bf Step 1. Diagonalization of \eqref{B-FD}.}
Let $\lambda_{\pm}(\xi)$ be the eigen values of  system \eqref{B-FD}. Analysis on the linear part of \eqref{B-FD} with $b=d>0$ yields that
\beq\label{eigen value}
\lambda_{\pm}(\xi)=\pm i\sqrt{\f{1-\gamma}{\gamma}}\f{A(\xi)^{\f12}(1-c\mu|\xi|^2)^{\f12}}{1+b\mu|\xi|^2}|\xi|,
\eeq
where $A(\xi)$ is a symbol of the Fourier multiplier $A(D)$ that is defined in \eqref{notation}.

Now, we  diagonalize the system \eqref{B-FD}.
Denoting by
\beno
\omega_1=\omega_1(\xi)=\f{1}{\gamma}\f{A(\xi)}{1+b\mu|\xi|^2},\quad \omega_2=\omega_2(\xi)=(1-\gamma)\f{1-c\mu|\xi|^2}{1+b\mu|\xi|^2}
\eeno
we have
\beq\label{H3}
\lambda_{\pm}(\xi)=\pm i\sqrt{\omega_1\omega_2}|\xi|.
\eeq

Letting
\beq\label{new unknowns}
W=|D|^{-1}\curl\vv v,\quad
Z_\pm=\z\pm\sqrt{\f{\omega_1(D)}{\omega_2(D)}}\f{1}{i|D|}\na\cdot\vv v,
\eeq
\eqref{B-FD} is equivalent to
\beq\label{H B-FD 1}\begin{aligned}
&\p_t W=0,\quad \p_tZ_\pm\pm i|D|\sqrt{\omega_1(D)\omega_2(D)}Z_\pm=f_\pm,
\end{aligned}\eeq
where
\beq\label{nonlinear for H}
f_\pm=\f{1}{\gamma}\f{\epsilon}{1-b\mu\D}\na\cdot(\z\vv v)\pm
\f{1}{2\gamma}\sqrt{\f{\omega_1(D)}{\omega_2(D)}}\f{i\epsilon|D|}{1-b\mu\D}(|\vv v|^2).
\eeq

\medskip

{\bf Step 2. Solutions to \eqref{H B-FD 1}.} Defining
\beno
W_0=|D|^{-1}\curl\vv v_0,\quad
Z_{\pm,0}=\z_0\pm\sqrt{\f{\omega_1(D)}{\omega_2(D)}}\f{1}{i|D|}\na\cdot\vv v_0,
\eeno
by virtue of Duhamel principle, the solutions to \eqref{H B-FD 1} are written as
\beq\label{H4}\begin{aligned}
&W(t,x)=W_0(x),\\
& Z_\pm(t,x)=e^{\mp it|D|\sqrt{\omega_1(D)\omega_2(D)}}Z_{\pm,0}(x)
+\int_0^te^{\mp i(t-s)|D|\sqrt{\omega_1(D)\omega_2(D)}}f_\pm(s,x)ds.
\end{aligned}\eeq

Thanks to \eqref{new unknowns}, we have
\beq\label{unknowns}
\z=\f12(Z_++Z_-),
\quad\vv v=\f{\na}{2i|D|}\sqrt{\f{\omega_2(D)}{\omega_1(D)}}(Z_+-Z_-)+\f{\na^\perp}{|D|}W,
\eeq
where $\na^\perp=(-\p_2,\p_1)^T$.

Since
\beno
A(\xi)=1-a\mu|\xi|^2+\f{1}{\gamma}\sqrt{\f{\mu}{\mu_2}}\sigma(\xi)
+\f{1}{\gamma^2}\f{\mu}{\mu_2}\sigma(\xi)^2,\quad \sigma(\xi)=\sqrt{\mu_2}|\xi|\coth(\sqrt{\mu_2}|\xi|),
\eeno
and
\beno
\lim_{|\xi|\rightarrow0}\sigma(\xi)=1,\quad\lim_{|\xi|\rightarrow\infty}\f{\sigma(\xi)}{\sqrt{\mu_2}|\xi|}=1,
\eeno
it is easy to check that
\beq\label{H5}\begin{aligned}
&\|\omega_1(\xi)\|_{L^\infty_\xi}+\|\omega_2(\xi)\|_{L^\infty_\xi}\lesssim1,
\quad\|\f{\omega_2(\xi)}{\omega_1(\xi)}\|_{L^\infty_\xi}+\|\f{\omega_1(\xi)}{\omega_2(\xi)}\|_{L^\infty_\xi}\lesssim1.
\end{aligned}\eeq
Actually, following similar derivation as $g(D)$, one could check that  $\omega_1(D)$, $\omega_2(D)$, $\sqrt{\f{\omega_1(D)}{\omega_2(D)}}$ and $\sqrt{\f{\omega_2(D)}{\omega_1(D)}}$ are zero-order pseudo-differential operators which satisfy
\beq\label{H6}\begin{aligned}
&\|\omega_1(D)f\|_{H^s}\sim\|f\|_{H^s},\quad\|\omega_2(D)f\|_{H^s}\sim\|f\|_{H^s}\\
&\|\sqrt{\f{\omega_2(D)}{\omega_1(D)}}f\|_{H^s}\sim\|f\|_{H^s},\quad\|\sqrt{\f{\omega_1(D)}{\omega_2(D)}}f\|_{H^s}\sim\|f\|_{H^s},
\end{aligned}\eeq
for any $s\in\R$ and $f\in H^s(\R^2)$.

By virtue of Plancherel theorem and \eqref{H5}, or due to \eqref{H6}, we deduce from \eqref{new unknowns} and \eqref{unknowns} that
\beq\label{equivalent energy for H}
\|W\|_{X^0_\mu}+\|Z_+\|_{X^0_\mu}+\|Z_-\|_{X^0_\mu}\sim\|\z\|_{X^0_\mu}+\|\vv v\|_{X^0_\mu}.
\eeq

\medskip

{\bf Step 3. The {\it a priori} energy estimate.} Thanks to \eqref{H4}, we first have
\beq\label{H7}
\|W\|_{X^0_\mu}=\|W_0\|_{X^0_\mu},
\eeq
\beq\label{H8}
\|Z_\pm(t)\|_{X^0_\mu}\lesssim\|Z_{\pm,0}\|_{X^0_\mu}+\int_0^t\|f_\pm(s)\|_{X^0_\mu}d\tau.
\eeq

Now, we derive the bound of $\|f_\pm\|_{X^0_\mu}\sim\|f_\pm\|_{L^2}+\sqrt\mu\|\na f_\pm\|_{L^2}$. Thanks to \eqref{nonlinear for H} and \eqref{H6}, we have
\beno\begin{aligned}
&\|f_\pm\|_{L^2}\lesssim\|\f{\epsilon}{1-b\mu\D}\na\cdot(\z\vv v)\|_{L^2}+
\|\f{\epsilon|D|}{1-b\mu\D}(|\vv v|^2)\|_{L^2}\\
&\lesssim\f{\epsilon}{\sqrt\mu}\bigl(\|\z\vv v\|_{L^2}+\||\vv v|^2\|_{L^2}\bigr)
\lesssim\f{\epsilon}{\sqrt\mu}\bigl(\|\vv v\|_{L^4}\|\z\|_{L^4}+\|\vv v\|_{L^4}^2\bigr).
\end{aligned}\eeno
Since Ladyzhenskaya's inequality yields
\beno
\|\vv v\|_{L^4}\lesssim\|\vv v\|_{L^2}^{\f12}\|\na\vv v\|_{L^2}^{\f12},\quad \|\z\|_{L^4}\lesssim\|\z\|_{L^2}^{\f12}\|\na\z\|_{L^2}^{\f12},
\eeno
we have
\beq\label{H9}
\|f_\pm\|_{L^2}\lesssim\bigl(\|\z\|_{L^2}+\|\vv v\|_{L^2}\bigr)\cdot\f{\epsilon}{\sqrt\mu}\bigl(\|\na\z\|_{L^2}+\|\na\vv v\|_{L^2}\bigr)
\lesssim \|\z\|_{X^0_\mu}^2+\|\vv v\|_{X^0_\mu}^2,
\eeq
provided that $\epsilon\sim\mu$.

Whereas thanks to \eqref{nonlinear for H} and \eqref{H6}, we have
\beno\begin{aligned}
&\|\na f_\pm\|_{L^2}\lesssim\|\f{\epsilon\na^2}{1-b\mu\D}(\z\vv v)\|_{L^2}+
\|\f{\epsilon|D|^2}{1-b\mu\D}(|\vv v|^2)\|_{L^2}\lesssim\f{\epsilon}{\mu}\bigl(\|\z\vv v\|_{L^2}+\||\vv v|^2\|_{L^2}\bigr).
\end{aligned}\eeno
After similar derivation as \eqref{H9}, if $\epsilon\sim\mu$, we have
\beq\label{H10}
\sqrt\mu\|\na f_\pm\|_{L^2}\lesssim \|\z\|_{X^0_\mu}^2+\|\vv v\|_{X^0_\mu}^2.
\eeq

Due to \eqref{H9} and \eqref{H10}, we have
\beq\label{H11}
\|f_\pm\|_{X^0_\mu}\lesssim \|\z\|_{X^0_\mu}^2+\|\vv v\|_{X^0_\mu}^2,
\eeq
which along with \eqref{H8} implies
\beq\label{H12}
\|Z_\pm(t)\|_{X^0_\mu}\lesssim\|Z_{\pm,0}\|_{X^0_\mu}+t\sup_{(0,t)}\bigl(\|\z\|_{X^0_\mu}+\|\vv v\|_{X^0_\mu}\bigr)^2.
\eeq

Thanks to \eqref{H7}, \eqref{H8}, \eqref{H12} and \eqref{equivalent energy for H}, there exist two constants $C_1>1$ and $C_2>1$ such that
\beq\label{priori for H}
\sup_{(0,t)}\bigl(\|\z\|_{X^0_\mu}+\|\vv v\|_{X^0_\mu}\bigr)\leq C_1\bigl(\|\z_0\|_{X^0_\mu}+\|\vv v_0\|_{X^0_\mu}\bigr)
+C_2t\sup_{(0,t)}\bigl(\|\z\|_{X^0_\mu}+\|\vv v\|_{X^0_\mu}\bigr)^2.
\eeq

\medskip

{\bf Step 4. Local existence and uniqueness.}
   By virtue of the {\it a priori } energy estimate \eqref{priori for H} and the standard contraction theorem, \eqref{B-FD}-\eqref{initial data} admits a unique solution $(\z,\vv v)$ on $[0,T]$ with $T\leq\f{1}{4C_1C_2(\|\z_0\|_{X^0_\mu}+\|\vv v_0\|_{X^0_\mu})}$ and
\beq\label{H29}
\max_{[0,T]}\bigl(\|\z\|_{X^0_\mu}+\|\vv v\|_{X^0_\mu}\bigr)\leq 2C_1\bigl(\|\z_0\|_{X^0_\mu}+\|\vv v_0\|_{X^0_\mu}\bigr).
\eeq
We omit the details here. One could refer to \cite{Hu}.

 \medskip

{\bf Step 5. Proof of the blow-up criteria \eqref{blow-up criteria}.} We prove \eqref{blow-up criteria} by contradiction. Assume that  when $T^*<\infty$, there holds
\beq\label{H27}
\liminf_{t\rightarrow T^*}\bigl(\|\z(t)\|_{X^0_\mu}+\|\vv v(t)\|_{X^0_\mu}\bigr)<\infty.
\eeq
Then there exists a constant $M>0$ such that for any $t\in[0,T^*)$
\beq\label{H28}
\|\z(t)\|_{X^0_\mu}+\|\vv v(t)\|_{X^0_\mu}\leq M.
\eeq

Taking $T_0=T^*-\f{1}{8C_1C_2M}$, by the definition of $T^*$, \eqref{B-FD} with initial data $(\z_0,\vv v_0)$ admits a unique solution
\beno
(\z,\vv v)\in C([0,T_0]; X^0_\mu(\R^2)\times X^0_\mu(\R^2)),
\eeno
and \eqref{H28} holds for $t\in[0,T_0]$. Now let $(\z(T_0),\vv v(T_0))$ be the new initial data. The local existence theorem proved in Steps 1-4 shows that \eqref{B-FD} with initial data $(\z(T_0),\vv v(T_0))$ admits a unique solution
\beno
(\z,\vv v)\in C([T_0,\bar{T}]; X^0_\mu(\R^2)\times X^0_\mu(\R^2)),
\eeno
where $\bar{T}=T_0+T_1$
\beno
T_1=\f{1}{4C_1C_2M}\leq \f{1}{4C_1C_2(\|\z(T_0)\|_{X^0_\mu}+\|\vv v(T_0)\|_{X^0_\mu})}.
\eeno
Then \eqref{B-FD} with initial data $(\z_0,\vv v_0)$ admits a unique solution on time interval $[0,\bar{T}]$ with
\beno
\bar{T}=T^*+\f{1}{8C_1C_2M}>T^*.
\eeno
This is a contradiction to the definition of $T^*$. Now, we complete the proof of the proposition.
\end{proof}

\subsection{Proof of Theorem \ref{Global existence for B-FD in case b=d}}

 We shall use  Proposition \ref{a priori for H B-FD} and the conservation law \eqref{conservation law} to establish the global theory for \eqref{B-FD} with $b=d>0$ and $a\leq0$, $c<0$.

Firstly, Proposition \ref{a priori for H B-FD} shows that \eqref{B-FD} with initial data $(\z_0,\vv v_0)$ admits a unique solutions
\beno
(\z,\vv v)\in C([0,T]; X^0_\mu(\R^2)\times X^0_\mu(\R^2)), \quad\hbox{for some } T>0.
\eeno

 Due to \eqref{Hamiltonian} and \eqref{sigma 1}, there exists a constant $c_1>0$ such that for any $t\in[0,T]$
\beq\label{H16}
\mathcal{H}(\z(t),\vv v(t))\geq\f{1-\gamma}{2}\|\z\|_{L^2}^2+\f{(1-\gamma)|c|}{2}\mu\|\na\z\|_{L^2}^2+\f{1}{2\gamma}\|\vv v\|_{L^2}^2
+c_1\mu\|\na\vv v\|_{L^2}^2-\f{1}{2\gamma}\epsilon\int_{\R^2}|\z||\vv v|^2dx.
\eeq

By virtue of Ladyzhenskaya's inequality,  there exists a constant $c_2>0$ such that
\beq\label{H15}\begin{aligned}
&\epsilon\int_{\R^2}|\z||\vv v|^2dx\leq\epsilon\|\z\|_{L^2}\|\vv v\|_{L^4}^2\leq c_2\epsilon\|\z\|_{L^2}\|\vv v\|_{L^2}\|\na\vv v\|_{L^2}
\leq\f12\|\vv v\|_{L^2}^2+\f12c_2^2\epsilon^2\|\z\|_{L^2}^2\|\na\vv v\|_{L^2}^2.
\end{aligned}\eeq

Noticing that $\mu\sim\epsilon$, we deduce from \eqref{H16} and \eqref{H15} that
\beq\label{H17}
\mathcal{H}(\z(t),\vv v(t))\geq c_3\bigl(\|\z\|_{L^2}^2+\mu\|\na\z\|_{L^2}^2+\|\vv v\|_{L^2}^2
+2\mu(1-\epsilon\|\z\|_{L^2}^2)\|\na\vv v\|_{L^2}^2\bigr),
\eeq
for some constant $c_3>0$.

Taking $\epsilon_0'=\f{1}{4\|\z_0\|_{L^2}^2}$, we have $\epsilon_0'\|\z_0\|_{L^2}^2=\f14<\f12$. Since solution $(\z,\vv v)$ is continuous in time, there exists $0<T_0<T$ such that
\beq\label{H31}
\epsilon_0'\|\z(t)\|_{L^2}^2<\f12\quad\hbox{for any } t\in[0,T_0],
\eeq
which along with \eqref{Hamiltonian} and \eqref{H17} implies that for any $\epsilon\leq\epsilon_0'$
and $t\in[0,T_0]$, the Hamiltonian $\mathcal{H}(\z(t),\vv v(t))$ is positive and
\beq\label{H17a}
\|\z(t)\|_{X^0_\mu}^2+\|\vv v(t)\|_{X^0_\mu}^2\leq\f{1}{c_3}\mathcal{H}(\z(t),\vv v(t))=\f{1}{c_3}\mathcal{H}(\z_0,\vv v_0).
\eeq

Taking $\epsilon_0=\min\{\f{c_3}{4\mathcal{H}(\z_0,\vv v_0)},\,\f{1}{4\|\z_0\|_{L^2}^2}\}\leq\epsilon_0'$, by virtue of \eqref{H31} and \eqref{H17a}, we  have for any $\epsilon\leq\epsilon_0$
and $t\in[0,T_0]$, $\mathcal{H}(\z(t),\vv v(t))$ is positive and
\beq\label{H32}
\epsilon_0\|\z(t)\|_{L^2}^2<\f12,\quad \|\z(t)\|_{X^0_\mu}^2+\|\vv v(t)\|_{X^0_\mu}^2\leq \f{1}{c_3}\mathcal{H}(\z_0,\vv v_0).
\eeq

Thus, \eqref{H32} and the blow-up criteria \eqref{blow-up criteria} in Proposition \ref{a priori for H B-FD} shows
that for any $\epsilon\in(0,\epsilon_0)$, the solution to \eqref{B-FD}-\eqref{initial data} can always be extended till $T^*=\infty$. Then \eqref{B-FD}-\eqref{initial data} admits a unique solution on time interval $[0,\infty)$ such that
\beq\label{H26}
\sup_{(0,\infty)}\bigl(\|\z(t)\|_{X^0_\mu}+\|\vv v(t)\|_{X^0_\mu}\bigr)\leq\f{1}{c_3}\mathcal{H}(\z_0,\vv v_0)\leq C\bigl(\|\z_0\|_{X^0_\mu}+\|\vv v_0\|_{X^0_\mu}\bigr).
\eeq
This is exactly \eqref{global estimate}. We complete the proof of  Theorem \ref{Global existence for B-FD in case b=d}.

\setcounter{equation}{0}
\section{Final comments}
So far we are not aware of  a global existence result  of {\it large} solutions to at least one of the Boussinesq-FD systems. Recall that for the Boussinesq $(abcd)$ systems such a result is only known in the one-dimensional $a=b=c=0, d=1$ case, see \cite{Sc, A} and the comments in the survey article \cite{Lannes3}.

Proving such a result for a Boussinesq -FD system is a challenging problem as is to prove the conjectured dichotomy for the life span $T_\epsilon$ of solutions : either $T_\epsilon =+\infty$, or $T_\epsilon=O(1/\epsilon).$
\section{Appendix}

{\bf 1. Proof of \eqref{equivalent 1}.} By the definition of $E_s(\vv V)$ in \eqref{E 1} and the expression of $S_{\vv V}(D)$ in \eqref{S for b neq d}, we  have
\beq\label{A 1}\begin{aligned}
&E_s(\vv V)=(1-\gamma)^2\gamma^2\|g(D)^{\f12}(1-b\mu\D)^{\f12}(1+c\mu\D)\Lambda^s\z\|_{L^2}^2\\
&\quad+\gamma(1-\gamma)\bigl((1-b\mu\D)\Lambda^s\vv v\,|\,(A(D)-\epsilon\z)(1+c\mu\D)\Lambda^s\vv v\bigr)_2\\
&\quad-\epsilon\gamma(1-\gamma)\bigl((1-b\mu\D)\Lambda^s\z\,|\,g(D)\bigl(\vv v\cdot(1+c\mu\D)\Lambda^s\vv v\bigr)\bigr)_2\\
&\quad-\epsilon\gamma(1-\gamma)\bigl((1-b\mu\D)\Lambda^s\vv v\,|\,g(D)\bigl(\vv v(1+c\mu\D)\Lambda^s\z\bigr)\bigr)_2\\
&\quad+\epsilon^2\sum_{i,j=1,2}\bigl((1-b\mu\D)\Lambda^sv^i\,|\,v^iv^j(g(D)-1)\Lambda^sv^j\bigr)_2\eqdefa A_1+A_2+A_3+A_4+A_5.
\end{aligned}\eeq

For $A_1$, using \eqref{g 1}, we have
\beq\label{A 2}
A_1\sim\|(1-b\mu\D)^{\f12}(1+c\mu\D)\Lambda^s\z\|_{L^2}^2\sim\|\z\|_{X^s_{\mu^3}}^2.
\eeq

For $A_2$, using the expression of $A(D)$ in \eqref{notation}, we first have
\beq\label{A 3}\begin{aligned}
&A_2=\gamma(1-\gamma)\bigl((1-b\mu\D)\Lambda^s\vv v\,|\,(1-\epsilon\z)(1+c\mu\D)\Lambda^s\vv v\bigr)_2
+\gamma(1-\gamma)|a|\mu\|(1-b\mu\D)^{\f12}(1+c\mu\D)^{\f12}\na\Lambda^s\vv v\|_{L^2}^2\\
&\quad+(1-\gamma)\sqrt{\f{\mu}{\mu_2}}\|\sigma(D)^{\f12}(1-b\mu\D)^{\f12}(1+c\mu\D)^{\f12}\Lambda^s\vv v\|_{L^2}^2
+\f{1-\gamma}{\gamma}\f{\mu}{\mu_2}\|\sigma(D)(1-b\mu\D)^{\f12}(1+c\mu\D)^{\f12}\Lambda^s\vv v\|_{L^2}^2\\
&\eqdefa A_{21}+A_{22}+A_{23}+A_{24}
\end{aligned}\eeq

Integrating by parts, we have
\beno\begin{aligned}
&A_{21}=\gamma(1-\gamma)\{\bigl(\Lambda^s\vv v\,|\,(1-\epsilon\z)\Lambda^s\vv v\bigr)_2+b|c|\mu^2\bigl(\D\Lambda^s\vv v\,|\,(1-\epsilon\z)\D\Lambda^s\vv v\bigr)_2\\
&\quad+(b-c)\mu\bigl(\na\Lambda^s\vv v\,|\,(1-\epsilon\z)\na\Lambda^s\vv v\bigr)_2-(b-c)\epsilon\mu\bigl((\na\z\cdot\na)\Lambda^s\vv v\,|\,\Lambda^s\vv v\bigr)_2\}
\end{aligned}\eeno
which along with \eqref{ansatz for amplitude} and \eqref{interpolation} implies that
\beno
A_{21}+(b-c)\epsilon\mu\bigl((\na\z\cdot\na)\Lambda^s\vv v\,|\,\Lambda^s\vv v\bigr)_2\sim\|\Lambda^s\vv v\|_{L^2}^2+\mu^2\|\D\Lambda^s\vv v\|_{L^2}^2+\mu\|\na\Lambda^s\vv v\|_{L^2}^2\sim\|\vv v\|_{X^s_{\mu^2}}^2.
\eeno
Using \eqref{ansatz for amplitude}, Sobolev inequality and H\"older inequality, we have
\beno
\epsilon\mu|\bigl((\na\z\cdot\na)\Lambda^s\vv v\,|\,\Lambda^s\vv v\bigr)_2|\leq\mu^{\f12}\epsilon\|\na\z\|_{L^\infty}\cdot\mu^{\f12}\|\na\Lambda^s\vv v\|_{L^2}\|\Lambda^s\vv v\|_{L^2}
\leq\mu^{\f12}\epsilon^{\f12}\bigl(\mu\|\na\Lambda^s\vv v\|_{L^2}^2+\|\Lambda^s\vv v\|_{L^2}^2\bigr).
\eeno
Then by virtue of \eqref{interpolation}, we have
\beq\label{A 4}
(1-\mu^{\f12}\epsilon^{\f12})\|\vv v\|_{X^s_{\mu^2}}^2\lesssim A_{21}\lesssim(1+\mu^{\f12}\epsilon^{\f12})\|\vv v\|_{X^s_{\mu^2}}^2.
\eeq

Due to \eqref{sigma 1}, we have
\beno\begin{aligned}
&A_{24}\gtrsim\mu\|\na(1-b\mu\D)^{\f12}(1+c\mu\D)^{\f12}\Lambda^s\vv v\|_{L^2}^2
\gtrsim\|\vv v\|_{X^s_{\mu^3}}^2-\mu\|\vv v\|_{X^s_{\mu^2}}^2,\\
&A_{24}\lesssim\mu\|(1-b\mu\D)^{\f12}(1+c\mu\D)^{\f12}\Lambda^s\vv v\|_{H^1}^2\lesssim\|\vv v\|_{X^s_{\mu^3}}^2,
\end{aligned}\eeno
which implies
\beq\label{A 5}
\|\vv v\|_{X^s_{\mu^3}}^2-\mu\|\vv v\|_{X^s_{\mu^2}}^2\lesssim A_{24}\lesssim \|\vv v\|_{X^s_{\mu^3}}^2.
\eeq
Similarly, using \eqref{interpolation} and \eqref{sigma 1}, we have
\beq\label{A 6}
\mu^{\f12}\||D|^{\f12}\vv v\|_{X^s_{\mu^2}}^2\lesssim A_{23}\lesssim\|\vv v\|_{X^s_{\mu^3}}^2.
\eeq

Using \eqref{interpolation}, we have
\beno
A_{22}\lesssim\|\vv v\|_{X^s_{\mu^3}}^2,
\eeno
which along with \eqref{A 3}, \eqref{A 4}, \eqref{A 5} and \eqref{A 6} implies
\beq\label{A 7}
\|\vv v\|_{X^s_{\mu^3}}^2+(1-\mu^{\f12}\epsilon^{\f12}-\mu)\|\vv v\|_{X^s_{\mu^2}}^2\lesssim A_2\lesssim\|\vv v\|_{X^s_{\mu^3}}^2.
\eeq

For $A_3$, $A_4$ and $A_5$, using \eqref{g 1} and \eqref{g 2}, we have
\beno
|A_3|+|A_4|+|A_5|\lesssim(1+\epsilon\|\vv v\|_{L^\infty})\epsilon\|\vv v\|_{L^\infty}\bigl(\|\z\|_{X^s_{\mu^2}}^2+\|\vv v\|_{X^s_{\mu^2}}^2\bigr),
\eeno
which along with \eqref{ansatz for amplitude} and Sobolev inequality implies
\beq\label{A 8}
|A_3|+|A_4|+|A_5|\lesssim\sqrt\epsilon\bigl(\|\z\|_{X^s_{\mu^3}}^2+\|\vv v\|_{X^s_{\mu^3}}^2\bigr).
\eeq

Thanks to \eqref{A 1}, \eqref{A 2}, \eqref{A 7} and \eqref{A 8}, we have
\beq\label{A 9}
(1-\mu-\sqrt\epsilon)\bigl(\|\z\|_{X^s_{\mu^3}}^2+\|\vv v\|_{X^s_{\mu^3}}^2\bigr)\lesssim E_s(\vv V)\lesssim\|\z\|_{X^s_{\mu^3}}^2+\|\vv v\|_{X^s_{\mu^3}}^2.
\eeq
 Taking $\epsilon$ and $\mu$ sufficiently small, we deduce from \eqref{A 9} that
\beq\label{A 10}
E_s(\vv V)\sim\|\z\|_{X^s_{\mu^3}}^2+\|\vv v\|_{X^s_{\mu^3}}^2.
\eeq
This is exactly \eqref{equivalent 1}.

\medskip

{\bf 2. Proof of \eqref{equivalent 2}.} By the definition of $E_s(\vv V)$ in \eqref{E 1} and the expression of $S_{\vv V}(D)$ in \eqref{S for b neq d}, noticing that in this case $g(D)=1-b\mu\D$, we  have
\beq\label{A 11}\begin{aligned}
&E_s(\vv V)=(1-\gamma)^2\gamma^2\|(1-b\mu\D)(1+c\mu\D)\Lambda^s\z\|_{L^2}^2\\
&\quad+\gamma(1-\gamma)\bigl((1-b\mu\D)\Lambda^s\vv v\,|\,(A(D)-\epsilon\z)(1+c\mu\D)\Lambda^s\vv v\bigr)_2\\
&\quad-\epsilon\gamma(1-\gamma)\bigl((1-b\mu\D)^2\Lambda^s\z\,|\,\vv v\cdot(1+c\mu\D)\Lambda^s\vv v\bigr)_2\\
&\quad-\epsilon\gamma(1-\gamma)\bigl((1-b\mu\D)\Lambda^s\vv v\,|\,(1-b\mu\D)\bigl(\vv v(1+c\mu\D)\Lambda^s\z\bigr)\bigr)_2\\
&\quad-b\mu\epsilon^2\sum_{i,j=1,2}\bigl((1-b\mu\D)\Lambda^sv^i\,|\,v^iv^j\D\Lambda^sv^j\bigr)_2\eqdefa A_1+A_2+A_3+A_4+A_5.
\end{aligned}\eeq
Following similar derivation as \eqref{equivalent 1}, under the assumption \eqref{ansatz for amplitude}, we have
 \beno\begin{aligned}
&A_1\sim\|\z\|_{X^s_{\mu^4}}^2,\quad\|\vv v\|_{X^s_{\mu^3}}^2+(1-\mu^{\f12}\epsilon^{\f12}-\mu)\|\vv v\|_{X^s_{\mu^2}}^2\lesssim A_2\lesssim\|\vv v\|_{X^s_{\mu^3}}^2\\
&|A_3|+|A_4|+|A_5|\lesssim\sqrt\epsilon\bigl(\|\z\|_{X^s_{\mu^4}}^2+\|\vv v\|_{X^s_{\mu^3}}^2\bigr)
\end{aligned}\eeno

Then for sufficiently small $\epsilon$ and $\eta$, we have
\beno
E_s(\vv V)\sim\mathcal{E}_s{(t)}\eqdefa\|\z(t)\|_{X^s_{\mu^4}}^2+\|\vv v(t)\|_{X^s_{\mu^{3}}}^2.
\eeno
This is exactly \eqref{equivalent 2}.

\medskip

{\bf 3. Proof of \eqref{equivalent 3}.} By the definition of $E_s(\vv V)$ in \eqref{E 3} and the expression of $S_{\vv V}(D)$ in \eqref{S for b=d>0}, we  have
\beq\label{A 12}\begin{aligned}
&E_s(\vv V)=(1-\gamma)\|(1+c\mu\D)^{\f12}\Lambda^s\z\|_{L^2}^2
+\bigl(\Lambda^s\vv v\,|\,(A(D)-\epsilon\z)\Lambda^s\vv v\bigr)_2\\
&\quad-\epsilon\bigl(\Lambda^s\z\,|\,\vv v\cdot\Lambda^s\vv v\bigr)_2-\epsilon\bigl(\Lambda^s\vv v\,|\,\vv v\Lambda^s\z\bigr)_2.
\end{aligned}\eeq
After similar derivation as \eqref{equivalent 1} and \eqref{equivalent 2}, we get that \eqref{equivalent 3} holds for sufficiently small $\epsilon$ and $\eta$.

\vspace{0.5cm}
\noindent {\bf Acknowledgments.} The first author was partially supported by the ANR Project ANuI. The work
of the second author was partially supported by NSF of China under grants 11671383 and
by an innovation grant from National Center for Mathematics and Interdisciplinary Sciences.

\end{document}